\documentclass[12pt]{amsart}

\usepackage{amssymb}

\usepackage{enumerate,color,mathrsfs}

\usepackage{hyperref,url}

\usepackage[all]{xy}

\renewcommand{\subset}{\subseteq}
\newtheorem{theorem}            {Theorem}[section]
\newtheorem{corollary}          [theorem]{Corollary}
\newtheorem{proposition}        [theorem]{Proposition}

\newtheorem{lemma}              [theorem]{Lemma}

\newtheorem{remark}             [theorem]{Remark}

\newcommand{\OS}[2]{\mathcal M_{{#2},{#1}}}
\newcommand{\OB}[1]{{\mathcal R_{#1}}}

\def\beq{ \begin{equation}}
\def\eeq{  \end{equation} }
\def\bes{\begin{equation*}}
\def\ees{\end{equation*}}
\def\ben{\begin{enumerate} }
\def\een{\end{enumerate} }
\def\benum{\begin{enumerate} }
\def\eenum{\end{enumerate} }
\def\bmat{\begin{bmatrix}}
\def\emat{\end{bmatrix}}
\def\barr{\begin{array}}
\def\earr{\end{array}}

\def\ss{\le} 

\def\cA{ {\mathcal A} }
\def\cB{ {\mathcal B} }

\def\cE{ {\mathcal E} }
\def\cF{ {\mathscr F} }

\def\cH{ {\mathcal H} }
\def\cI{ {\mathcal I} }

\def\cK{ {\mathcal K} }
\def\cL{ {\mathcal L} }
\def\cM{{\mathcal M}}

\def\cQ{ {\mathcal Q} }
\def\cR{{ \mathcal R }}
\def\cS{{\mathcal S} }

\def\cZ{ {\mathcal Z} }

\def\Mnns{\CC^{n\times n}}

\def\Mpps{\CC^{p\times p}}
\def\BL{\mathcal B_L}

\def\bBL{\partial \BL}

\def\cB{\mathcal B}

\def\Bll{\cB_{\ell' , \ell} }

\def\Mdd{\CC^{d' \times d} }
\def\Mll{\CC^{\ell' \times \ell} }

\def\ss{\le} 

\def\T{T}

\def\la{\lambda}

\def\T{\ast}

\def\NN{\mathbb N}

\def\nonsingular{nondegenerate }

\def\tg{g}

\newcommand{\CC}{{\mathbb C}}
\newcommand{\RR}{{\mathbb R}}
\newcommand{\DD}{{\mathbb D}}

\numberwithin{equation}{section}

\title[Noncommutative pencil ball maps]
{Analytic mappings between noncommutative
  pencil balls}

\author[Helton]{J. William Helton${}^1$}
\address{J. William Helton, Department of Mathematics\\
  University of California \\
  San Diego}
\email{helton@math.ucsd.edu}
\thanks{${}^1$Research supported by NSF grants
DMS-0700758, DMS-0757212, and the Ford Motor Co.}

\author[Klep]{Igor Klep${}^2$}
\address{Igor Klep, Univerza v Ljubljani, Fakulteta za matematiko in fiziko \\
and
Univerza v Mariboru, Fakulteta za naravoslovje in matematiko 
}
\email{igor.klep@fmf.uni-lj.si}
\thanks{${}^2$Research supported by the Slovenian Research Agency grants 
J1-3608 and P1-0222.}

\author[McCullough]{Scott McCullough${}^3$}
\address{Scott McCullough, Department of Mathematics\\
  University of Florida 
   }
   \email{sam@math.ufl.edu}
\thanks{${}^3$Research supported by the NSF grant DMS-0758306.}

\subjclass[2000]{Primary 47A56, 46L07; Secondary 32H99, 32A99, 46L89}
\date{\today}
\keywords{noncommutative analytic function, complete isometry, ball map,
linear matrix inequality}

\begin{document}

\begin{abstract}
In this paper,
we analyze problems involving matrix variables
for which we use a noncommutative algebra setting.
 To be more specific, we use a class of functions (called NC analytic
 functions)  defined by
power series in noncommuting variables and
evaluate these functions on sets of matrices of all dimensions;
 we call such situations  dimension-free.
These types of functions have recently been used
 in the study of dimension-free linear
system engineering problems \cite{HMPVieee,deOHMPima}.

In the earlier paper \cite{HKMS} we characterized NC analytic  maps that send
dimension-free matrix balls to di\-men\-sion-free matrix
balls  and carry the boundary to the boundary;
such maps we call ``NC ball maps''.
In this paper  we turn to a more general  dimension-free ball
$\cB_L$, called a ``pencil ball",
associated with a homogeneous linear pencil
$$L(x):= A_1 x_1 + \cdots + A_\tg x_\tg, \quad A_j\in\CC^{d'\times d}.$$
For
  $X=\mbox{col}(X_1,\dots,X_{\tg}) \in (\Mnns)^{\tg},$
  define
  $
  L(X):=\sum A_j \otimes X_j
$
and let
$$\cB_L:=\big(\{ X\in (\Mnns)^{\tg}: \|L(X)\| < 1\}\big)_{n\in\NN}.$$
We study the generalization of NC ball maps to
 these pencil balls $\cB_L$, and call them ``pencil ball maps''.
We show that every $\cB_L$ has a minimal dimensional (in a
certain sense) defining pencil $\tilde L$.
Up to normalization, a pencil ball map
is the direct sum of $\tilde L$
with an
NC analytic map of the pencil ball into the ball.
That is,  pencil ball maps are simple, in contrast
to the classical result of D'Angelo \cite[Chapter 5]{dAn}
showing there is a great variety of such analytic maps
from $\CC^g$ to $\CC^m$ when $g \ll m$.
To prove our main theorem,
this paper uses the results
of our previous paper \cite{HKMS} plus entirely different techniques,
namely, those of completely contractive maps.  

What we do here is a small piece of the bigger puzzle
of understanding how Linear Matrix Inequalities (LMIs) behave with respect
to noncommutative change of variables.
\end{abstract}
\maketitle

\section{Introduction}
  Given positive integers $n,d,d^\prime$ and $\tg$,
  let $\Mdd$ denote the $d^\prime \times d$ matrices
  with complex coefficients and $(\Mnns)^{\tg}$
  the set of $\tg$-tuples of $n\times n$ matrices.
  For $A_1,\dots,A_{\tg} \in \Mdd$,  the expression
\beq\label{eq:trulyPencil}
  L(x)=\sum_{j=1}^{\tg} A_j x_j,
\eeq
  is a {\bf homogeneous linear pencil}. (Often the term linear
  pencil refers to an \emph{affine} linear function; i.e.,
  a sum of a constant term plus a homogeneous linear pencil.)
  Given
  $X=\mbox{col}(X_1,\dots,X_{\tg}) \in (\Mnns)^{\tg},$
  define
\beq\label{eq:trulyPencilTensor}
  L(X)=\sum_{j=1}^\tg A_j \otimes X_j.
\eeq

  Let $\BL(n)=\{ X\in (\Mnns)^{\tg}: \|L(X)\| < 1\}$
  and let $\BL$ denote the sequence $(\BL(n))_{n\in\NN}$.  Similarly, let
  $\OS{\ell}{\ell^\prime} = ((\Mnns)^{\ell^\prime\times \ell})_{n\in\NN}$.
  The main result of this  paper describes analytic
  mappings $f$ from the {\bf pencil ball} $\BL$ to $\OS{\ell}{\ell^\prime}$
  that  preserve the boundary
  in the sense described
  at the end of Section \ref{sec:formal} below.

  In the remainder of this introduction,
  we give the definitions and background
  necessary for a precise statement of the result,
  and provide  a guide
  to the body of the paper.

\subsection{Formal Power Series}
 \label{sec:formal}
  Let $x=(x_1,\dots,x_{\tg})$  be a $\tg$-tuple
  of noncommuting indeterminates and let
  $\langle x\rangle$ denote the set of all {\bf words} in $x$. This includes
  the empty word denoted by $1$.
  The length of a word $w\in\langle x\rangle$ will be denoted $|w|$.
  For an abelian group $R$ we use $R\langle x\rangle$ to denote the abelian
  group of all (finite)
  sums of {\bf monomials} (these are
  elements of the form $r w$ for $r\in R$ and $w\in \langle x\rangle$).

  Given positive integers $\ell,\ell^\prime$, a {\bf formal
  power series} $f$ in $x$ with
  $\Mll$ coefficients is an expression
  of the form
 \begin{equation}
  \label{eq:def-f}
   f = \sum_{m=0}^\infty \sum_{\substack{\; w\in\langle x\rangle\\ |w|=m}} f_w w= \sum_{m=0}^\infty f^{(m)},
 \end{equation}
  where $f_w \in \Mll$ and $f^{(m)}\in \Mll\langle x\rangle$ is the {\bf homogeneous component} of degree $m$ of $f$, that is, the sum of all monomials in $f$ of degree $m$.
  For $X\in (\Mnns)^\tg$, $X=\mbox{col}(X_1,\dots,X_\tg)$
  and a word
 \bes
  w=x_{j_1}x_{j_2}\cdots x_{j_m}\in\langle x\rangle,
 \ees
  let
 \bes
   w(X) = X_{j_1}X_{j_2} \cdots X_{j_m}\in\Mnns.
 \ees
  Define
 \bes
    f(X)=\sum_{m=0}^\infty \sum_{\substack{\; w\in\langle x\rangle\\ |w|=m}}  f_w \otimes w(X),
 \ees
  provided the series converges (summed in the indicated order).
  The function $f$ is {\bf analytic} \index{analytic}
on $\BL$
  if for each  $n$ and $X\in \BL(n)$ the series $f(X)$
  converges.
  Thus, in this case,  the formal power series $f$
  determines a mapping from $\BL(n)$ to
  $(\Mnns)^{\ell^\prime \times \ell}$ for each $n$
  which is expressed by writing $f:\BL \to \OS{\ell}{\ell^\prime}$.

  The analytic function $f:\BL \to \OS{\ell}{\ell^\prime}$
   is {\bf contraction-valued} \index{contractive}\index{contraction-valued}
   if    $\|f(X)\|\le 1$
   for each $X\in \BL$;  i.e., if the
   values of $f$ are contractions.
   Let $\bBL(n)$ denote the set of all
   $X\in (\Mnns)^\tg$ with $\|L(X)\|=1$.
   If $f:\BL \to \OS{\ell}{\ell^\prime}$ is contraction-valued
   and $X\in\bBL(n)$,
   then, by Fatou's Theorem,  the analytic function
   $f_X:\mathbb D\to (\Mnns)^{\ell^\prime \times \ell}$
   defined by $f_X(z)=f(zX)$ has boundary
   values almost everywhere; i.e.,
   $f(\exp(it)X)$ is defined for almost
   every $t$.
   (We use $\mathbb D$ to denote the unit disc
   $\{z\in\CC : |z|<1\}$.)
   The contraction-valued function
   $f$ is a {\bf pencil ball map}
   \index{binding} \index{ball map}\index{pencil ball map} if
   $\|f((\exp(it)X)\|=1$  a.e.~for every $X\in \bBL$.
   Here the boundary $\bBL$ of the pencil ball $\cB_L$ is
   the sequence $(\bBL(n))_{n\in\NN}$.

\subsection{The Main Result}
 \label{sec:main}
  The homogeneous linear pencil
  $L$ is {\bf nondegenerate}, if
  it is one-one in the sense that 
\beq\label{eq:weLoveTheReferee}
\forall X\in(\Mnns)^{g}: \; \big( L(X)=0 \; \Rightarrow\;
X=0\big)
\eeq
for all $n\in\NN$.

\begin{lemma}\label{lem:nonsingular}
  For a homogeneous linear pencil $L(x)=\sum_{j=1}^{g} A_j x_j$ the following
 are equivalent:
\ben[\rm (i)]
 \item $L$ is nondegenerate;
 \item $L(X)=0$ implies $X=0$ for all $X\in\CC^{g}$, i.e., condition
{\rm\eqref{eq:weLoveTheReferee}} holds for $n=1$;
 \item the set $\{A_j : j=1,\ldots,g\}$ is linearly independent.
\een
\end{lemma}

  \begin{proof}
(i) $\Rightarrow$ (ii) and (ii) $\Rightarrow$ (iii) are obvious.
For the remaining implication (iii) $\Rightarrow$ (i),
note that $L(X)$ equals $\sum X_j\otimes A_j$ modulo
the canonical shuffle. If this expression equals $0$, then the linear
dependence of the $A_j$ (applied entrywise) implies $X=0$.
  \end{proof}

\def\tL{\tilde{L}}
\def\tA{\tilde{A}}
   The homogeneous linear pencil
 \bes
   \tL(x) =\sum_1^{\tg} \tA_j x_j,
 \ees
   is {\bf equivalent} to $L$ if
   $\|\tL(X)\|=\|L(X)\|$
   for every $n$ and $X\in (\Mnns)^\tg$, i.e., $\cB_{L}=\cB_{\tL}$.

A $\ell'\times \ell$ \nonsingular homogeneous linear pencil $\tL$ is called {\bf a
minimal dimensional defining pencil for a ball} if
$\cB_{L}=\cB_{\tL}$ for some $d'\times d$ homogeneous linear
\nonsingular pencil $L$ implies $\ell'\ell\leq d'd$.
Equivalently: if $L$ is equivalent to $\tL$, then $d'd\geq \ell'\ell$.
Two equivalent minimal dimensional defining pencils for a ball
are
the same up to normalization, cf.~Corollary \ref{cor:minDim}
and Theorem \ref{thm:bigDeal}.  Thus, while there are potentially many
ways of defining minimal, all will agree
with the current usage.  Indeed, heuristically any condition which
eliminates redundant or simply irrelevant summands from the pencil $L$
should do. 

\begin{lemma}\label{lem:minDim}
Suppose $L$ is a \nonsingular homogeneous linear pencil and
$\tL$ is a minimal dimensional defining pencil for $\cB_L$.
Then there is a homogeneous linear pencil $J$ and unitaries $Q,G$
satisfying
$$L=Q(\tL \oplus J)G^*.$$
\end{lemma}

The proof of this lemma can be found in Section 5.1; it is an
immediate consequence of Theorem \ref{thm:bigDeal} and Corollary
\ref{cor:minDim}.

 The following theorem is the main result of this paper.

 \begin{theorem}
  \label{thm:main}
     Suppose $L$ is a \nonsingular homogeneous linear pencil and
     $\tL$ is a minimal dimensional defining pencil for $\cB_L$.
If $f:\BL \to \OS{\ell}{\ell^\prime}$ is a pencil ball map
     with $f(0)=0$,
     then    there is a contraction-valued analytic
     $\tilde{f}:\BL \to \OS{m}{m^\prime}$ such that
 \begin{equation*}
    f(x) =U \begin{pmatrix} \tL(x) & 0\\ 0&\tilde{f}(x)  \end{pmatrix} V^*
 \end{equation*}
 for some $m',m\in\NN_0$ and
 unitaries $U\in \CC^{\ell'\times \ell'}$ and $V\in\CC^{\ell\times \ell}$.
 \end{theorem}

\begin{corollary}\label{cor:mainn}
     Suppose $L$ is a \nonsingular homogeneous linear pencil and
     $\tL$ is a $\ell'\times \ell$
     minimal dimensional defining pencil for $\cB_L$.
If $f:\BL \to \OS{\ell}{\ell^\prime}$ is a pencil ball map
     with $f(0)=0$, then
     \begin{equation*}
    f(x) =U\tL(x)  V^*
 \end{equation*}
 for some
 unitaries $U\in \CC^{\ell'\times \ell'}$ and $V\in\CC^{\ell\times \ell}$.
\end{corollary}

\subsection{Related work}
 An elegant theory of noncommutative analytic functions is
   developed in the articles \cite{KVV} and \cite{Voi1,Voi2};
   see also \cite{Pop1, Pop6, Pop3, Pop2, Pop4, Pop5}.
   What we use in this article are specializations of definitions
     of these papers.
Also there are results on various classes of noncommutative
functions on several NC domains as we now describe.

\ben[\rm (a)]
\item
For the special case of the NC ball of row
contractions, Corollary \ref{cor:mainn} is the same as Corollary 1.3
in \cite{HKMS} and appears in a  weaker form
in Popescu's paper \cite{Pop5}. Namely, the assumption that the map
takes the boundary of the NC ball into the boundary of the NC ball is
replaced in \cite{Pop5} by the stronger assumption that the NC
ball map is biholomorphic.

\item
Also related to the current paper is  \cite{AK}
which studies the special cases of noncommutative polydisks (and
noncommutative poly-halfplanes)
rather than general pencil balls. With this caveat in mind, the
part of \cite{AK} most closely related to the author's paper is the 
 study noncommutative analytic functions
  given by a formal power series with
$\CC^{q \times q}$ coefficients mapping an $N$-tuple
of $n \times n$  unitary matrices to a $qn \times  qn$ unitary matrix
for
each $n\in\NN$.
This class generalizes the notion of inner function
(e.g., for the case $g = d = d' = 1$ and $L(x) = x$
any scalar inner function is in the \cite{AK} class).
Our pencil ball maps have the additional property
that $f(B)$ is an $n \times n$ matrix of
norm $1$ for every $n \times n$ matrix $B$  of norm $1$
and as we shall see this forces $f$ to be only a single Blaschke
factor.
\item
The paper \cite{BGM} obtains realization results for contraction-valued
noncommutative analytic functions defined on pencil balls of a special
form. In the multiplicity-one case, the coefficients of the pencil are all
partial isometries whose set of initial spaces (overlaps of initial spaces
for different coefficients allowed) form a direct sum decomposition
for the whole domain space, and similarly for the final spaces.
\end{enumerate}

\subsection{The Linear Term and Complete Contractivity}
\label{sec:linterm}
  The first step in the proof of Theorem \ref{thm:main}
  is an analysis of the linear term of $f$ which
  is $f^{(1)}=\sum_{j=1}^\tg  f_{x_j}x_j$ with $f$ expressed
  as in equation \eqref{eq:def-f}.
  It involves heavily the theory of completely contractive
  and completely positive linear maps, the nature
  of which   we illustrate later in this introduction by discussing
  Theorem \ref{thm:main} in the case of linear maps.
 The books \cite{Pau,Pi,BlM}
 provide comprehensive introductions
  to the theory of operator systems, spaces, and algebras,
  and completely contractive and completely isometric
  mappings. The
  papers \cite{BH1} and \cite{BH2} treat very generally
  complete isometries into a $C^*$-algebra.
  Indeed, the results here in the linear case are
  very similar to those in \cite{BL,BH1,BH2} when
  specialized to our setting.    Namely,   the injective envelope of a subspace of
  $m\times n$ matrices is, up to equivalence, a direct sum of full
  matrix spaces.  Rather than only quoting existing results from the
  literature, we have chosen to make the presentation self contained
  with an exposition of the operator space background at the
  level of generality needed for the present purposes.  We believe
  this makes the article more accessible and may expose a wider
  audience to the utility of the extensive literature in operator
  spaces and systems.

  Let  $F$ be a subspace of $\Mdd$.
  For each $n$, there is the subspace
  $F_n = F\otimes \Mnns$ of $\Mdd \otimes \Mnns$.
  A linear mapping $\varphi:F\to \Mll$
  induces linear mappings
  $\varphi_n:F\otimes \Mnns\to \Mll\otimes \Mnns$
  by $\varphi_n(e\otimes Y)=\varphi(e)\otimes Y.$
  Write $F\ss \OS{d}{d^\prime}$ to indicate
  that we  are identifying $F$ with the sequence $(F_n)_{n\in\NN}$.
  A {\bf completely contractive} mapping $\varphi:F\to \OS{\ell}{\ell^\prime}$
  is a linear mapping $\varphi:F\to \Mll$
  such that  $\|\varphi_n(Z)\|\le \|Z\|$ for every
  $n$ and $Z\in F\otimes \Mnns$.
  \index{completely contractive}
  If instead, $\|\varphi_n(Z)\|= \|Z\|$ for every
    $n$ and $Z\in F\otimes \Mnns$, then $\varphi$ is
{\bf completely isometric}.
We shall be interested in completely isometric maps
 acting on the {\bf range} $\OB{L}$ of a homogeneous linear pencil $L$.
 This range for a $d' \times d$ pencil in $\tg$ variables
 is defined as follows.
  For each $n$, let $\OB{L}(n)$ denote the
range of $L$ applied to $X$ in $(\Mnns)^\tg$:
$$\OB{L}(n)
= \{ L(X): \ X \in (\Mnns)^\tg \} .$$
  Let $\OB{L}=(\OB{L}(n))_{n\in\NN}.$
In particular, 
we have $\OB{L}=F \ss \OS{d}{d^\prime}$.

\begin{theorem}
 \label{thm:linear-main}
 Let $L$ be a \nonsingular homogeneous linear pencil and
     $\tL$ a minimal dimensional defining pencil for $\cB_L$.
  If $\psi:\OB{L}\to \OS{\ell}{\ell^\prime}$
  is completely isometric, then there exist unitaries $U\in\CC^{\ell'\times \ell'}$
  and $V\in\CC^{\ell\times \ell}$, positive integers $m,m^\prime$,
  and a completely contractive
  mapping $\phi:\OB{\tL}\to \OS{m}{m^\prime}$
  such that
 \begin{equation*}
   \psi(L(x))= U \begin{pmatrix} \tL(x) & 0 \\ 0 & \phi(\tL(x))\end{pmatrix}V^*.
 \end{equation*}
\end{theorem}

 If $f:\BL \to \OS{\ell}{\ell^\prime}$ is a pencil ball map with
  $f(0)=0$ and if $L$ is a \nonsingular homogeneous linear pencil,
  then the linear part $f^{(1)}$ of $f$
  induces  a completely isometric mapping
  $$\psi:\OB{L}\to\OS{\ell}{\ell^\prime},\quad L(x)\mapsto f^{(1)}(x)$$
  (see Lemma \ref{lem:isoDerivative})
  to which Theorem \ref{thm:linear-main} applies.

\subsection{Readers Guide}
  The remainder of the paper is organized as follows.
  Section \ref{sec:prelims} provides background
  on completely contractive maps, operator spaces,
  and injective envelopes needed for the remainder
  of the paper. The point of departure is
  Arveson's extension theorem.
   In Section \ref{sec:ci-corners} a result
   from our previous paper \cite{HKMS} is used to characterize
   completely isometric mappings from
   $\oplus \OS{d_j}{d_j^\prime}$ to $\OS{n}{n^\prime}$
   and in Section \ref{sec:envelope-RL} the
   injective envelope of  $\OB{L}$ is determined.
   The main results, Theorems \ref{thm:linear-main}
   and \ref{thm:main} are proved in Section \ref{sec:cimaps}.

   The main results naturally generalize to
   mappings $f:\mathcal B_{L}\to \cB_{L'}$;
   i.e., analytic mappings between pencil balls
   determined by homogeneous linear (nondegenerate)
   pencils. The details are in Section \ref{sec:more}.

\section{The Complete Preliminaries}
 \label{sec:prelims}

  For the reader's convenience, in this section we
  have gathered the background material on
  completely positive, contractive, and isometric mappings
  and injective envelopes needed in the sequel.
  Since our attention is restricted to the case of
  finite dimensional subspaces of matrix algebras,
  the exposition here is considerably more concrete than that found in
  the literature, where the canonical level of generality
  involves arbitrary subspaces of  $C^*$-algebras. 
 We have followed the general outline, based upon
  the off-diagonal techniques of Paulsen, found
  in his excellent book,  cf.~\cite[p.~98]{Pau}.
  The reader with
  expertise in operator spaces and systems
  could likely skip or skim this section and proceed to Section
  \ref{sec:ci-corners}.  The reader  familiar with the
  work of Blecher and Hay  in \cite{BH1,BH2} or of Blecher and
  Labuschagne in \cite{BL} might safely skip to Section
  \ref{sec:cimaps}.    Indeed, this section and the next two
  establish the fact that the injective envelope of a subspace of
  $m\times n$ matrices is, up to equivalence, a direct sum of full
  matrix spaces - see Theorem \ref{thm:structureCI}.

\subsection{Completely Contractive and Completely Positive Maps}
 \label{sec:prelimsO}
Recall the notion of completely contractive defined in
Section \ref{sec:linterm}. We shall also need a  related notion
of positivity.
  A subspace  $S\ss \OS{m}{m}$
  is an {\bf operator system} if it is self-adjoint (that is,
  closed under $X\mapsto X^*$)
  and contains the identity.
  A {\bf completely positive}
  mapping $\psi:S\to \OS{p}{p}$
  is a linear map $\psi:S\to \Mpps$ such that
  $\psi_n(Z)$ is a positive semi-definite
  for every $n$ and every positive semi-definite
  matrix $Z\in S\otimes \Mnns.$ \index{completely positive}
  \index{operator system}

  The significance of completely positive maps is that, while
  positive maps $S\to \Mpps$ do not necessarily extend
  to positive maps on all of  $\CC^{m\times m}$, \emph{completely} positive
  maps on $S$ \emph{do} extend to completely positive maps on all of $\OS{m}{m}$.

\begin{theorem}[Arveson's extension theorem, cf.~{\protect\cite[Theorem 7.5]{Pau}}]
 \label{thm:arveson}
  If $S\ss \OS{m}{m}$ is an operator system and
  $\psi:S\to \OS{p}{p}$ is completely positive, then
  there is a completely positive mapping
  $\Psi:\OS{m}{m} \to \OS{p}{p}$ such that
  $\psi=\Psi|_S$.
\end{theorem}

  There are numerous connections between completely
  contractive and completely positive maps of which the off-diagonal
  construction will be used repeatedly.
  Given $F\ss \OS{d}{d'}$, let
$$
  \cS_F = 
  \begin{pmatrix} \CC I_d & F^* \\ F & \CC I_{d'} \end{pmatrix}
 = \{ \begin{pmatrix} \lambda I_d & b^* \\ a & \eta I_{d'} \end{pmatrix} : \lambda, \eta \in \mathbb C,
 \ \ a,b\in F \} \ss \OS{d+d^\prime}{d+d^\prime}.
$$
 Evidently $\cS_F$ is an operator system.

  Given a completely contractive $\varphi:F\to\OS{\ell}{\ell^\prime}$, the
  mapping $\psi:\cS_F\to \OS{\ell+\ell^\prime}{\ell+\ell^\prime}$ defined by
\beq
\label{eq:corner}
  \psi(\begin{pmatrix} \lambda I_d & b^* \\ a & \eta I_{d'} \end{pmatrix})
    = \begin{pmatrix} \lambda I_\ell & \varphi(b)^* \\ \varphi(a)& \eta I_{\ell'} \end{pmatrix}
\eeq
 is completely positive.  It is also {\bf unital} in that  $\psi(I)=I$.
 \index{unital}
 The converse is true, too.
 The compression of a completely positive unital 
  $\psi:\cS_F\to \OS{\ell+\ell^\prime}{\ell+\ell^\prime}$ to
 its lower left-hand corner is completely contractive.

\subsection{The injective envelope}
   Determining the structure of completely isometric mappings
   is intimately intertwined with the notion of the injective
   envelope.

   A subspace  $E\ss \OS{d}{d^\prime}$
   is {\bf injective} \index{injective} if whenever $F\ss \OS{\ell}{\ell^\prime}$
   and $\psi:F\to E$ is completely contractive, then there
   exists a completely contractive $\tilde\psi:\OS{\ell}{\ell^\prime} \to E$ such that
    $\tilde\psi|_F=\psi$:
\[ \xymatrix{
 F
\ar@{^{(}-{>}}[rr]
\ar[d]^{\psi}_{\text{c.contr.}} && \OS{\ell}{\ell^\prime}
\ar@{{}{--}>}[dll]_{\exists \tilde\psi}^{\text{c.contr.}} \\
E
}
\]

  The notion of injective envelope is categorical.

 \begin{lemma}
  \label{lem:nonsense}
   If $E\ss \OS{m}{m^\prime}$ is injective
   and $\tau:E\to\OS{\ell}{\ell^\prime}$ is
   completely isometric, then $\tau(E)$ is injective.
 \end{lemma}

 \begin{proof}
   It is enough to observe that the mapping $\tau^{-1}:\tau(E)\to E$
   is completely isometric. For the details,
   suppose $F\ss \OS{d}{d^\prime}$ and $\varphi:F\to \tau(E)$
   is completely contractive. Then $\tau^{-1}\circ \varphi:F\to E$
   is completely contractive and hence extends to
   a completely contractive $\tilde{\varphi}:\OS{d}{d^\prime}\to E$.
   The mapping $\tau \circ \tilde{\varphi}:\OS{d}{d^\prime}\to\tau(E)$ is then
   a completely contractive extension of $\varphi$.
 \end{proof}

  \begin{lemma}
    The spaces $\OS{\ell}{\ell^\prime}$ are injective.
  \end{lemma}

 \begin{proof}
   View $\OS{\ell}{\ell^\prime}$ as the lower left-hand
   corner of $\OS{\ell+\ell^\prime}{\ell+\ell^\prime}$ as
   in equation \eqref{eq:corner}.
   Note that the mapping
   $\Psi:\OS{\ell+\ell^\prime}{\ell+\ell^\prime} \to \OS{\ell}{\ell^\prime}$
  given by
\beq\label{eq:cornerProj}
  \Gamma (\begin{pmatrix} x_{11} & x_{12} \\ a & x_{22} \end{pmatrix})
    =a
\eeq
  is completely contractive.
  Likewise, the inclusion $\iota$ of $\OS{d}{d^\prime}$
  into $\OS{d+d^\prime}{d+d^\prime}$ as the lower
  left-hand corner is completely contractive.

   Suppose $F\ss \OS{d}{d^\prime}$ and  $\varphi:F\to \OS{\ell}{\ell^\prime}$
   is completely contractive.  Then $\psi:F\to \OS{\ell+\ell^\prime }{\ell+\ell^\prime}$,
   as in equation \eqref{eq:corner}, is completely positive.
   By Theorem \ref{thm:arveson}, $\psi$ extends
   to a completely positive $\Psi:\OS{d+d^\prime}{d+d^\prime}\to\OS{\ell+\ell^\prime}{\ell+\ell^\prime}$.
      The mapping $\Gamma \circ \Psi \circ \iota$ is a
   completely contractive extension of $\varphi$.
 \end{proof}

   A mapping $\Phi:\OS{d}{d'} \to J\ss \OS{d}{d'}$ is
   a {\bf projection} provided $\Phi$ is onto and $\Phi\circ\Phi=\Phi$.

  \begin{proposition}\label{prop:injProj}
    A subspace $E\ss \OS{d}{d^\prime}$ is injective if and only
    if there is a completely contractive projection $\Phi:\OS{d}{d^\prime} \to E$.
  \end{proposition}

  \begin{proof}
    If $E$ is injective, apply the definition with $F=E\leq
    \OS{d}{d^\prime}$ and $\psi$ the identity map. The conclusion
    is there exists a completely contractive
    $\Phi$ which extends $\psi$ and maps onto (into)
    $E$. It follows that $\Phi$ is a projection (completely contractive).

\[ \xymatrix{
 E
\ar@{^{(}-{>}}[rr]
\ar[d]_{\rm id} && \OS{d}{d'}
\ar@{{}{--}>}[dll]_{\exists \Phi}^{\text{c.contr.}} \\
E
}
\]

    Conversely, suppose the completely contractive $\Phi$ exists and
    let $F,\OS{\ell}{\ell^\prime}$ and $\psi$ be as in the definition
    of an injective space.
    Since $\OS{d}{d^\prime}$ is injective and $E\ss \OS{d}{d^\prime}$,
    there exists a completely contractive extension $\tilde{\psi}$ of
    $\psi$ mapping $\OS{\ell}{\ell^\prime}$ into $\OS{d}{d^\prime}$.  The
    composition $\varphi=\Phi \circ\tilde{\psi}$ is completely contractive
    into $E$ and extends $\psi$:
\[ \xymatrix{
 F
\ar@{^{(}-{>}}[rr]
\ar[d]_{\psi} && \OS{\ell}{\ell'} \ar[d]^{\exists \tilde{\psi}}
\ar@{{}{--}>}[dll]_{\Phi\circ\tilde{\psi}} \\
E && \OS{d}{d'} \ar[ll]^{\Phi}
}
\]
  \end{proof}

    Suppose $F\ss E\ss \OS{d}{d^\prime}$.
    A completely contractive mapping
    $\rho:E\to E$ is an
    {\bf $F$-map} on $E$ \index{F-map}
     provided $\rho(f)=f$ for all $f\in F$.
    For $E$ injective with $F\ss E\ss \OS{d}{d^\prime}$,
define a partial order $\leq_E$ on $F$-maps on $E$ by $\sigma \le_E \rho$
    if $\|\sigma(x)\| \le \|\rho(x)\|$ for each $x\in E$.
    An $F$-map on $E$ which is minimal with respect
   to this ordering is an {\bf $E$-minimal $F$-map}.
   \index{minimal $F$-map}

  \begin{lemma}
     Given $F\ss E\ss \OS{d}{d^\prime}$,
     the collection of $F$-maps on $E$  is compact.
     Further, if $E$ is injective,
     then there is an $E$-minimal $F$-map.
   Indeed, for any $F$-map $\rho$ on $E$ there exists an $E$-minimal
     $F$-map $\sigma$ such that $\sigma \le_E \rho$.
  \end{lemma}

  \begin{proof}
    Let $\mathcal F$ denote the collection of all $F$-maps on $E$.
    Since the inclusion $F\to \OS{d}{d^\prime}$
    is completely contractive and $\OS{d}{d^\prime}$
    is injective, the set of all $F$-maps is nonempty.
    The collection $\mathcal F$ is evidently closed
    and bounded and therefore compact.

    To establish the existence of an $E$-minimal $F$-map,
    apply Zorn's lemma. Namely, given a decreasing net $(\rho_\lambda)$
    there exists, by compactness, a convergent subnet $(\rho_\eta)$
    converging to some $\kappa$ (this does not depend upon
    decreasing).  Because the original net was decreasing, it follows
    that $\kappa \le \rho_\lambda$ for all $\lambda$. Hence
    every decreasing chain has a lower bound. An application
    of Zorn's lemma produces a minimal element.
  \end{proof}

 \begin{lemma}
  \label{lem:minimal-projection}
   If $\sigma$ is an  $E$-minimal $F$-map, then
   $\sigma(\sigma(x))=\sigma(x)$ for $x\in E$.
 \end{lemma}

 \begin{proof}
   Let $\psi_n$ be the average of the first $n$ powers of $\sigma$,
  \begin{equation*}
     \psi_n(x)=\frac{1}{n} \sum_{j=1}^n \sigma^{j}(x).
  \end{equation*}

   Then $\psi_n$ is an $F$-map and $\|\psi_n(x)\|\le \|\sigma(x)\|$ for
   all $x\in E$,
   so by minimality $\|\psi_n(x)\| =\|\sigma(x)\|$. The same argument
   shows $\|\sigma(\sigma(x))\|=\|\sigma(x)\|$.  Thus,
      \[
      \|\sigma(x-\sigma(x))\|= \|\psi_n(x-\sigma(x))\|
               = \frac{1}{n}\|\sigma(x)-\sigma^{n+1}(x)\|
               \le \frac{2}{n} \|x\|
               \]
    for all $n$. Hence $\sigma(x)=\sigma(\sigma(x))$.
 \end{proof}

  Let $F\ss \OS{d}{d^\prime}$ be given.
  An $E$ such that $F\ss E\ss \OS{d}{d^\prime}$
  is a {\bf concrete injective envelope} \index{injective envelope}
  \index{concrete injective envelope} of $F$ if $E$ is injective
  and if $J$ is injective with $F\ss J \ss E$, then
  $J=E$.

 \begin{theorem}
   Each $F\ss \OS{d}{d^\prime}$ has a concrete injective envelope.
   In fact, the range of each $\OS{d}{d^\prime}$-minimal $F$-map $\sigma$
   is a concrete injective envelope for $F$.
 \end{theorem}

 \begin{proof}
   Let $\sigma$ be a given $\OS{d}{d^\prime}$-minimal $F$-map 
   and let $E$ denote the range of $\sigma$.
   By Proposition \ref{prop:injProj} and
   Lemma \ref{lem:minimal-projection}, $E$ is injective.
   Suppose $J$ is
   also injective and    $F\ss J\ss E$. There is
   a completely contractive projection $\psi:\OS{d}{d^\prime} \to J$
   (onto).  It follows that $\|\psi(\sigma(x))\|\le \|\sigma(x)\|,$
   and therefore $\|\psi(\sigma(x))\| = \|\sigma(x)\|$
   by minimality of $\sigma$.
   Thus $\psi|_E:E\to J\subseteq E$ is isometric and,
   since the spaces are finite dimensional, $J=E$.

 \end{proof}

 \begin{lemma}
 \label{lem:cinjenv}
   Suppose $E$ is a concrete injective envelope of $F\ss\OS{d}{d^\prime}.$
   If  $\sigma$ is an $E$-minimal $F$-map,
   then $\sigma$ is the identity. 
 \end{lemma}

 \begin{proof}
    Let $J=\sigma(E).$ Then $F\ss J\ss E$.
    From Lemma \ref{lem:minimal-projection}
    $\sigma: E \to J$ is a completely
    contractive projection. There is a completely contractive
    projection $\Phi: \OS{d}{d^\prime} \to E$ and so $\sigma \circ \Phi$
    is a completely contractive projection $\OS{d}{d^\prime} \to J.$
    Hence $J$ is injective. By minimality
    $J=E$.  By Lemma \ref{lem:minimal-projection},
    $\sigma$ is the identity.
 \end{proof}

 \begin{lemma}
  \label{lem:F-map-auto-identity}
   Suppose $E$ is a concrete injective envelope of $F\ss\OS{d}{d^\prime}$.
   If $\rho$ is an $F$-map on $E$, then $\rho$ is the identity.
 \end{lemma}

 \begin{proof}
   Choose a minimal $F$-map $\sigma\le \rho$.
   By the previous lemma, $\sigma$ is the identity on $E$. Hence,
   for $y\in E$, we find
   $\|y\|=\|\sigma(y)\|\le \|\rho(y)\|\le \|y\|$ and therefore
   $\rho$ is an $E$-minimal $F$-map.  Another application of
   the previous lemma shows $\rho$ is the identity.
  \end{proof}

 \begin{corollary}
  \label{cor:injective-envelope-unique}
   If $E_1,E_2 \ss \OS{d}{d^\prime}$ are both concrete injective
   envelopes for $F$, then there exists a completely isometric isomorphism
   $\phi:E_1\to E_2$.
 \end{corollary}

 \begin{proof}
    Since $E_1$ is injective, apply the definition of injective
    to $F\ss E_2$ and the inclusion mapping of $F$ into
    $E_1$ produces a completely contractive $\varphi:E_2\to E_1$
    which is the identity on $F$. Similarly there exists a
    completely contractive $\psi:E_1\to E_2$ which is the identity
    on $F$. The composition $\psi\circ\varphi$ is then an $F$-map on $E_1$
    and is therefore the identity mapping by Lemma \ref{lem:F-map-auto-identity}.
 \end{proof}

 \begin{remark}\rm
   The corollary allows one to define {\bf the} concrete injective envelope of $F$;
   i.e., it is unique in the category whose objects are $F\ss \OS{d}{d^\prime}$
   and whose morphisms are completely contractive maps.
 \end{remark}

The following corollary plays an essential role in what follows.

  \begin{corollary}
   \label{cor:essential}
    Let $E$ be a concrete injective envelope of $F\ss \OS{d}{d^\prime}$.
    If $\psi: E\to \OS{\ell}{\ell^\prime}$ is completely contractive
    and $\psi|_F$ is completely isometric,  then $\psi$
    is completely isometric.  Moreover, no proper super-space of $E$
    has this property.
  \end{corollary}

  \begin{proof}
      Let $F^\prime=\psi(F)$.  Because $\psi$ is a complete isometry,
      $F^\prime\ss \OS{\ell}{\ell^\prime}$.  Let $\phi:F^\prime \to E$ denote
      the mapping $\psi(f)\mapsto f$ which is completely contractive
      (actually isometric). Since $E$ is injective, it follows 
      that there is a completely contractive $\tau: \OS{\ell}{\ell^\prime} \to E$
      extending $\phi$. The composition $\rho=\tau\circ\psi$ is
      an $F$-map on $E$  (completely contractive and the identity on $F$).
      Thus $\rho$ is  the identity on $E$ 
      by Lemma \ref{lem:F-map-auto-identity}.
      It follows that $\psi$ must be completely
      isometric.

      For the second statement, given $E<J\ss\OS{\ell}{\ell'}$, by injectivity
      the identity map $E\to E$ extends to a completely contractive map
      $J\to E\ss J$ which is clearly not completely isometric.
  \end{proof}

 \begin{remark}\label{rem:ess}\rm
    This corollary says that studying completely isometric mappings
    on $F$ is the same as studying completely isometric mappings
    on an injective envelope $E$ of $F$, since any completely
    isometric mapping $F\to \OS{\ell}{\ell^\prime}$ extends to a
    completely isometric mapping $E\to\OS{\ell}{\ell^\prime}$.
 \end{remark}

 \subsection{Operator Systems and Injectivity}
   Recall
the notion of an operator system $\cS$ defined in
Section \ref{sec:prelimsO}.
 Because a
   unital map $\Phi:\OS{m}{m}\to \cS$ is
   completely contractive if and only if it is
   completely positive, an operator system
   is injective (as an operator space) if and only
   if it is the range of a completely
   positive projection (which is automatically unital).

 \begin{proposition}[Choi-Effros, cf.~{\protect\cite[Theorem 15.2]{Pau}}]
    An injective operator system $\cS\ss\OS{m}{m}$ is completely isometrically isomorphic
    to a $C^*$-algebra under the multiplication $a\circ b= \Phi(ab)$,
    where $\Phi$ is a given completely positive projection from $\OS{m}{m}$
    onto $\cS$.
 \end{proposition}

\subsection{Injective envelopes and corners of matrix algebras}
  A subspace $F\ss\OS{d}{d^\prime}$ naturally embeds
  in the operator system
$$
  \mathcal S_F = \{ \begin{pmatrix} \lambda I_d & b^* \\ a & \eta I_{d'} \end{pmatrix} : \lambda, \eta \in \mathbb C,
 \ \ a,b\in F \} \ss \OS{d+d^\prime}{d+d^\prime}.
$$
  Let $\Gamma:\OS{d+d^\prime}{d+d^\prime} \to \OS{d}{d^\prime}$
  denote the completely contractive projection onto the lower left-hand corner; i.e.,
  \beq\label{eq:cornerProjj}
  \Gamma(\begin{pmatrix} x_{11} & x_{12} \\ a & x_{22}\end{pmatrix}) = a.
 \eeq

\begin{proposition}
 \label{prop:from-system-to-space}
  Suppose $E\le\OS{d^\prime}{d}$ is  an injective operator space.
   There exists subspaces 
  $\cA$ and $\cB$ of $\OS{d}{d}$ and $\OS{d^\prime}{d^\prime}$ respectively
  such that the concrete  injective envelope $\cE$ of $S_E$ has the form
\[
 \cE =\{ \begin{pmatrix} a & f^* \\ e & b \end{pmatrix} : a\in \cA, \
 \ b\in \cB, \ \ e, f\in E\}.
\]
\end{proposition}

\begin{proof}
  Let $n=d+d^\prime$.  Let $\cE\le \OS{n}{n}$ be an injective
  envelope of $S_E$. 
  There is a completely contractive unital projection $\Psi:
  \OS{n}{n}\to \cE$.  By Stinespring's Theorem [Pau, Theorem 4.1], there exists a (finite
  dimensional) Hilbert space $\cH$,  a representation
  $\pi:\OS{n}{n} \to \mathcal B(\cH)$, and an isometry $V:\mathbb C^n\to
  \cH$ such that
\[
  \Psi(x) = V^* \pi(x) V.
\]
  Consider the $n\times n$ matrices
\[
   p_1 =\begin{pmatrix} I_d & 0\\0 & 0 \end{pmatrix}, \ \
   p_2=\begin{pmatrix} 0 & 0 \\ 0 & I_{d^\prime} \end{pmatrix}.
\]
   It follows that $P_j=\pi(p_j)$ are projections with $P_1+P_2=I$.
   Hence we can decompose $\cH = \cH_1\oplus \cH_2$ with $\cH_j$ the
   range of $P_j$.  With respect to this decomposition, and the
   natural decomposition of $\mathbb C^n =\mathbb C^d \oplus
   \mathbb C^{d^\prime},$ express $V$ as $V=(V_{j,k})$. Since
   $p_1\in S_E\le \cE$ we have
\[
 \Psi(p_1)=\begin{pmatrix} I_{d} & 0 \\ 0 & 0 \end{pmatrix} 
   = V^* \begin{pmatrix} I_{\cH_1} & 0 \\ 0 & 0\end{pmatrix} V.
\]
  It follows that $V^*_{12} V_{12}=0$. Hence $V_{12}=0$.  A similar argument shows
  $V_{21}=0$.   

  Let $V_j = P_jV$.  At this point we have 
\[
  \Psi(x) = (V_1+V_2)^* \pi(x) (V_1+V_2).
\] 
  Let $\cA_{j,k}$ denote the range of the mapping $x\mapsto V_j^* \pi(p_j  x
  p_k)V_k$.  Since $V_j^* \pi( p_j x p_k) V_k= \Psi(p_j x p_k)\in
  \cE$, it follows that $\cA_{j,k}\subset \cE$.   On the other hand,
 any $x$ can be written as $\sum p_j x p_k$  and  $V_\ell^* \pi(p_j x
 p_k) V_m=0$ unless $(\ell,m)=(j,k)$. Hence  $\cE$ has the
 desired representation.  

  Now, with $\Gamma$ as in (2.3),  the mapping $\Delta=\Gamma|_{\cE}:\cE\to \cA_{2,1}$ is a completely contractive
  projection onto $\cA_{2,1}$ and thus $\cA_{2,1}$ is an injective
  operator space.   Since $S_E\subset \cE$, it follows that
  $E\subset \cA_{2,1}$.  To see that $E=\cA_{2,1}$, note
  that there is a completely contractive projection $\varphi:\cA_{2,1}\to
  E$. Hence the mapping $\Phi:S_{\cA_{2,1}}\to S_{\cA_{2,1}}$ defined by
\[
  \begin{pmatrix} \lambda  & f^* \\ e & \mu \end{pmatrix} 
    \mapsto \begin{pmatrix} \lambda & \varphi(f)^* \\
       \varphi(e) &  \mu \end{pmatrix}
\] 
  is unital and completely positive. 
  Thus, because $\cE\supseteq S_{\cA_{2,1}}$ is a concrete injective envelope of $S_E$, the map
  $\Phi$ extends uniquely as a completely contractive map on
  $S_{\cA_{2,1}}$.  Since $\Phi$  is also the identity on $S_E$ it
  must be the identity on $\cA_{2,1}$. 
  The conclusion  $\cA_{2,1}=E$ follows.  $\square$
\end{proof}

\section{Completely isometric maps $\oplus \OS{d_j}{d_j^\prime}\to\OS{n}{n^\prime}$}
 \label{sec:ci-corners}

  In this section we classify completely isometric maps $\Psi$
   from a direct sum of matrix spaces into matrices,
   which is  a special case of our main result appearing
     in Section \ref{sec:cimaps} below.
   More precisely,  we consider
   a completely isometric
   $\Psi:\oplus_j \OS{d_j}{d_j^\prime}\to\OS{n}{n^\prime}$
   for $d_j,d_j',n,n'\in\NN$.

  We note that the results here, and in the next section, are very much in the spirit of,
  and flow from the same considerations, and can be made to follow
  from the results, as in \cite{BH1,BH2,BL} when specialized to our
  concrete  setting.

\begin{theorem}
 \label{thm:structureCI}
  Let $d=\sum d_j$, $d'=\sum d_j'$.
    If $\Psi:\oplus_j \OS{d_j}{d_j^\prime}\to\OS{n}{n^\prime}$ is completely
   isometric, then there exists a completely contractive mapping
   $\psi:\oplus_j \OS{d_j}{d_j^\prime}\to\OS{n-d}{n^\prime-d'}$ and
   unitaries $U:\mathbb C^{n'}\to \CC^{n'}$ and $V:\CC^{n}\to\CC^{n}$
    such that
 \begin{equation}\label{eq:repCompIso}
    \Psi(x)= U \begin{pmatrix} x & 0 \\ 0 & \psi(x) \end{pmatrix} V^*.
 \end{equation}
\end{theorem}

  Two operator systems $\mathcal S, \mathcal S^\prime \ss \OS{m}{m}$
  are equal, up to unitary equivalence, if there exists
  an $m\times m$ unitary matrix $U$ such that
 \begin{equation*}
  \mathcal S =\{ U T U^* : T\in\mathcal S^\prime \}.
 \end{equation*}

\begin{corollary}
\label{cor:structureIS}
     If $\cS \ss \OS{m}{m}$ is an injective operator system, then there exist
    integers $n_j$ and a completely contractive unital mapping
    $\Phi_0: \oplus_j \OS{n_j}{n_j} \to \OS{m-\sum n_j}{m-\sum n_j}$ so that,
    up to unitary equivalence,
  \begin{equation}\label{eq:injOpSysRep}
     \cS=\{ \begin{pmatrix} a & 0 \\ 0 & \Phi_0(a) \end{pmatrix} : a \in \oplus_j \OS{n_j}{n_j}\}.
  \end{equation}

   Further,
   if $\tau:\cS\to \OS{p}{p}$ is completely isometric
   and unital, then the range
   of $\tau$ is also an injective operator system and
   $\tau$ is, up to unitary equivalence, of the form,
  \begin{equation}
   \tau(\begin{pmatrix} a & 0 \\ 0 & \Phi_0(a) \end{pmatrix})
       = \begin{pmatrix} a & 0 \\ 0 & \tau_0(a)\end{pmatrix},
  \end{equation}
    for some completely contractive unital $\tau_0$.
\end{corollary}

 \begin{proof}
   Since $\cS$, being injective, is completely isometrically
   isomorphic to a finite dimensional $C^*$-algebra, there exist (finitely many) integers $n_j$
   and a completely isometric unital mapping $\Psi:\oplus_j\OS{n_j}{n_j} \to \cS\ss
   \OS{m}{m}$.
     By Theorem \ref{thm:structureCI}, there are unitaries
   $U$ and $V$ and a completely contractive mapping
   $\psi$ such that
 $$
  \Psi(x) = U \begin{pmatrix} x & 0 \\ 0 & \psi(x) \end{pmatrix} V^*.
 $$
   Since $\Psi$ is unital, we may assume $V=U$. Hence, up to unitary
   equivalence, $\cS$ being the image of $\Psi$, is of the desired form
   \eqref{eq:injOpSysRep}.

   If $\tau$ is completely isometric and $\cS$ is injective, then $\tau(\cS)$
  is injective by Lemma \ref{lem:nonsense}.

   To prove the last part, observe that the mapping
   $\tilde{\tau}:\oplus_j \OS{n_j}{n_j}\to \OS{p}{p}$
   defined by
 \begin{equation*}
  \tilde{\tau}(a) = \tau (\Psi(a))
 \end{equation*}
   is completely isometric. Hence, by the first part of the corollary
   there exists a unitary $W$ and completely contractive
   $\tau_0$ such that
 \begin{equation*}
   \tau(\Psi(a))= W \begin{pmatrix} a & 0\\ 0 & \tau_0(a) \end{pmatrix} W^*.
   \qedhere
 \end{equation*}
 \end{proof}

\begin{proof}[Proof of Theorem {\rm \ref{thm:structureCI}}]
      Let $\Psi_j$ denote the restriction of $\Psi$ to the $j$-th coordinate.
      Thus, $\Psi_j: \OS{d_j}{d_j^\prime} \to \OS{n}{n^\prime}$
      is completely isometric.
      From our earlier results (cf.~\cite[Theorem 1.3]{HKMS}),
      there exist unitaries
      $U_j: \mathbb C^{n^\prime}\to \mathbb C^{n^\prime}$ and
      $V_j:\mathbb C^{n}\to \mathbb C^{n}$  such that
  \begin{equation}
   \label{eq:repPsij1}
    \tau_j(x):=
       U_j^* \Psi_j(x)V_j = \begin{pmatrix} x & 0 \\ 0 & \psi_j(x) \end{pmatrix}
  \end{equation}
    for some completely contractive
    $\psi_j:\OS{d_j}{d_j^\prime} \to \OS{n-d_j}{n^\prime -d_j^\prime}$.
    These $\Psi_j$ need to
    fit together in a way which keeps $\Psi(x)=\sum \Psi_j(x_j)$ completely
    isometric.

   We now decompose each of $\mathbb C^n$ and $\mathbb C^{n^\prime}$
    compatibly with the $\Psi_j$ and $\tau_j$ as follows using the notations
    and orthogonal decomposition of equation \eqref{eq:repPsij1}.
    In particular, $\tau_j(x)$ decomposes as a direct sum mapping
    $\mathbb C^{d_j}\oplus \mathbb C^{n-d_j} \to \mathbb
    C^{d_j^\prime}\oplus \mathbb C^{n^\prime-d_j^\prime}$.  We let
    $\mathcal I_j$ and $\mathcal I_j^\prime$ denote the first
    summands, $\mathbb C^{d_j}$ and $\mathbb C^{d_j^\prime},$
    respectively.   We further decompose the second summands, $\mathbb
    C^{n-d_j}$ and $\mathbb C^{n^\prime-d_j^\prime}$ as follows. 
    Let $\mathcal Z_j^\prime$ denote the subspace
    $\{z\in\mathbb C^{n^\prime -d_j^\prime}: z^* \psi_j(x)=0 \text{ for all } x\}$.
    Let $\mathcal C_j^\prime$ denote the orthogonal complement of $\mathcal Z_j^\prime$
    in $\mathbb C^{n^\prime -d_j^\prime}$.  Similarly, let
    $\mathcal Z_j$ denote $\{z \in \mathbb C^{n-d_j}: \psi_j(x)z=0 \text{ for all } x\}$
    and let $\mathcal C_j$ denote the complement of $\mathcal Z_j$
    in $\CC^{n-d_j}$. 
   Thus, we have $\mathbb C^n= \mathcal I_j \oplus \mathcal C_j \oplus
   \mathcal Z_j$ and likewise for the primes. 

   The claim is:
\ben[\rm (a)]
\item
     $V_j \mathcal I_j \perp V_k \mathcal  C_k$;
\item
     $V_j \mathcal I_j \perp V_k \mathcal I_k$ for  $j\ne k$;
\item
     $U_j \mathcal I_j^\prime \perp U_k \mathcal C_k^\prime$;
\item
$U_j\mathcal I_j^\prime \perp U_k \mathcal I_k'$ for $j\ne k$.
\een

Clearly, (a) and (c) hold for $j=k$. For the rest of the proof
we assume, without loss of generality, that $j=1$ and $k=2$.

(a) $\land$ (b):
  Let $v\in \mathcal I_1$ and $u\in\mathcal I_1^\prime$ be given unit vectors.
    Let  $x_1 = u_1 v^* $ and let $x_2\in\OS{d_2}{d_2^\prime}$ of norm one be given.
     Both $x_1$ and $x_2$ have norm one and hence $x=x_1\oplus \exp(it) x_2 \oplus 0\in\oplus_j\OS{d_j}{d_j'}$
     also has norm one. Thus,
 \begin{equation*}
  \begin{split}
      1\ge \| \Psi(x)V_1v \|& = \| \Psi_1(x_1) V_1v +\exp(it) \Psi_2(x_2)V_1 v\| \\
          &= \| U_1 \begin{pmatrix} x_1 &0\\0 & \psi_1(x_1)\end{pmatrix} \begin{pmatrix} v \\ 0 \end{pmatrix}
        + \exp(it)U_2 \begin{pmatrix} x_2 &0\\0&\psi_2(x_2)\end{pmatrix} V_2^* V_1 v \|\\
        &= \| U_1 \begin{pmatrix} u_1 \\ 0 \end{pmatrix}
        + \exp(it)U_2 \begin{pmatrix} x_2 &0\\0&\psi_2(x_2)\end{pmatrix} V_2^* V_1 v\|.
   \end{split}
 \end{equation*}
  Since
 \begin{equation*}
   1= \| U_1 \begin{pmatrix} u_1 \\ 0 \end{pmatrix} \|,
 \end{equation*}
  it follows that
 \begin{equation*}
    \begin{pmatrix} x_2 &0\\0&\psi_2(x_2)\end{pmatrix} V_2^* V_1 v =0
 \end{equation*}
  for all $x_2$.  
  $V_2^* V_1 v\in \mathcal Z_2.$
  Equivalently,
  $V_2^* V_1 v \in (\mathcal I_2\oplus \mathcal C_2)^\perp$ which is
  equivalent to the claim.

(c) $\land$ (d):
  This argument is similar to that above.
   We claim that
   $U_1 \mathcal I_1'$ is orthogonal to $U_2
   (\mathcal I_2' \oplus \mathcal C_2')$.
   Note, if $0\neq \gamma \in \mathcal I_2' \oplus \mathcal C_2'$, then
   there is an $x_2$ so that
   $$\begin{pmatrix}x_2^*&0\\
   0&\psi_2(x_2)^*\end{pmatrix}  \gamma \ne 0.$$
   Given a unit vector
   $v\in \mathcal I_1'$, let $\gamma$ denote the
   projection of $U_2^* U_1 v$ onto $\mathcal I_2' \oplus \mathcal C_2'$.
   Choose a unit vector $u_1 \in\mathbb C^{d_1}$, let
   $x_1=v u_1^*$, and let $x_2\in \OS{d_2}{d_2^\prime}$ of norm one be given.
   Let $x=x_1\oplus \exp(it) x_2 \oplus 0$ and estimate,
 \begin{equation*}
  \begin{split}
     1\ge \|\Psi(x)^* U_1 v\| &=
     \| \Psi_1(x_1)^* v +\exp(it) \Psi_2(x_2)^* U_1 v\| \\
     &=
     \| V_1 \begin{pmatrix} x_1^* &0\\0&\psi_1(x_1)^*\end{pmatrix}
            \begin{pmatrix} v \\ 0 \end{pmatrix}
              + \exp(it)
V_2 \begin{pmatrix} x_2^* &0\\0&\psi_2(x_2)^*\end{pmatrix}
              U_2^*U_1 v \| \\
           &= \| V_1             \begin{pmatrix} u_1 \\ 0 \end{pmatrix}
              + \exp(it)V_2 \begin{pmatrix} x_2^* &0\\0&\psi_2(x_2)^*\end{pmatrix}
              \gamma \|.
  \end{split}
 \end{equation*}
    It follows that $\Psi_2(x_2)^* \gamma=0$ over all choices of $x_2$.
    Hence $\gamma \in \cZ_2'$; but also
    $\gamma \in \mathcal I_2' \oplus \mathcal C_2'$.
     Thus, $\gamma=0$. Since this is true for all $v$, the conclusion follows
    and the claim is proved.

     Let $\mathcal M = \oplus V_j \mathcal I_j$.
    Note that $\cM$ has dimension $d=\sum d_j$ and
    $\cM^\perp$ has dimension $n-d$. 
   Before proceeding, we make the following observations.  If
   $j\ne k$ and $\gamma_j\in \mathcal I_j$, then $V_k^* V_j
   \gamma_j=0$.  Further, since $V_k\mathcal I_k \subset \mathcal M,$
   if $\Gamma \in \mathcal M^\perp$, then there is a
   $\delta_k\in\mathcal I_k^\perp$ such that $\gamma=V_k
   \delta_k$. Thus, for $x_k\in \OS{d_k}{d_k^\prime},$
 \[
   U_k \tau_k(x_k) V_k^* \Gamma = U_k \psi_k(x_k) \delta_k \in (\mathcal I_k^\prime)^\perp.
 \]
   On the other hand, $\psi_k(x_k)\delta_k$ is also orthogonal to
   $(\mathcal Z_k^\prime)^\perp$ (by definition).  We conclude that
   $U_k \tau_k(x_k)V_k^* \Gamma$ is in $U_k \mathcal C_k^\prime$ and thus
\begin{equation}
  \label{eq:ortho}
    U_k \tau_k(x_k) V_k^* \Gamma \in (\mathcal M^\prime)^\perp.
\end{equation}

   Define
    $V^* : \CC^n \to \CC^n$ by
 \begin{equation*}
  \begin{split}
   V^* : \cM \oplus \cM^\perp & \to
        \oplus \CC^{d_j} \oplus \CC^{n-d}, \\
 \begin{pmatrix} \oplus V_j\gamma_j \\ V_0 \gamma\end{pmatrix} & \mapsto
  \begin{pmatrix} \oplus \gamma_j \\ V \gamma \end{pmatrix},
  \end{split}
 \end{equation*}
   where $V_0^* : \cM^\perp \to \CC^{n-d}$ is (any) unitary.

   Similarly, define $U: \CC^{n^\prime}\to \CC^{n^\prime}$ by
 \begin{equation*}
  \begin{split}
 U:\oplus \CC^{d_j^\prime} \oplus \CC^{n^\prime-d^\prime} &
    \to \cM^\prime \oplus (\cM^\prime)^\perp, \\
  \begin{pmatrix} \oplus \delta_j \\ \delta\end{pmatrix} & \mapsto
   \begin{pmatrix}  \oplus U_j \delta_j \\ U_0 \delta \end{pmatrix},
  \end{split}
 \end{equation*}
   where $U_0$ is (any) unitary from $\CC^{n^\prime-d^\prime}$ to
  $(\cM^\prime)^\perp$ and $\cM'=\oplus U_j\cI_j'$.

  We record the following observation, which follows from \eqref{eq:ortho}.
  Given $x_k \in \OS{d_k}{d_k^\prime}$,
$\oplus \gamma_j \in \oplus \CC^{d_j}$ and $\gamma\in \CC^{n-d}$,
 \begin{equation*}
   U_k \tau_k(x_k) V_k^*
           \begin{pmatrix}  \oplus V_j \gamma _j \\ V_0 \gamma\end{pmatrix}
    = U_k \tau_k(x_k) V_k^*(\sum_j V_j \gamma_j +V_0 \gamma)
    = U_k x_k \gamma_k \oplus \psi_k(x_k)V_k^* V_0 \gamma,
 \end{equation*}
   where the orthogonality respects the decomposition $\mathcal
   M^\prime \oplus (\mathcal M^\prime)^\perp$. 

   To finish the proof, let $\oplus x_j \in \oplus \OS{d_j}{d_j^\prime}$
  and $\oplus \gamma_j \in \oplus \CC^{d_j}$ and $\gamma\in \CC^{n-d}$
  be given. Then,
 \begin{equation*}
  \begin{split}
     U^* \Psi(\oplus x_k)V  V^*
        \begin{pmatrix} \oplus V_j \gamma_j \\ V_0 \gamma \end{pmatrix}
        &=  U^* \sum_k \Psi_k(x_k)
          \begin{pmatrix} \oplus V_j \gamma_j  \\ V_0 \gamma\end{pmatrix}
       = U^* \sum_{k} U_k \tau_k(x_k)V_k^*(\sum_j V_j\gamma_j
           +V_0\gamma) \\
     &= U^* \sum U_k (x_k \gamma_k +\psi_k(x_k) V_k^*V_0\gamma)
     =  \begin{pmatrix} \oplus  x_k \gamma_k \\
         U_0^* \sum U_k \psi_{k}(x_k) V_k^* V_0 \gamma \end{pmatrix}\\
         &= \begin{pmatrix} x & 0 \\ 0 & \psi(x) \end{pmatrix}
          \begin{pmatrix} \oplus \gamma_j\\ \gamma\end{pmatrix} 
                      = \begin{pmatrix} x & 0 \\ 0 & \psi(x) \end{pmatrix}
          V^*  \begin{pmatrix} \oplus V_j \gamma_j\\ V_0 \gamma\end{pmatrix}
               \end{split}  
 \end{equation*}
for the completely contractive $\psi(x)= U_0^* \sum U_k \psi_{k}(x_k) V_k^* V_0$.
\end{proof}

\section{The Injective Envelope of $\OB{L}$}
 \label{sec:envelope-RL}

  The following theorem exposes the structure of
  a concrete injective envelope of $\OB{L}$, for
  a \nonsingular homogeneous linear pencil $L$.
  It will be applied, along with Theorem
  \ref{thm:structureCI}, to prove
  Theorem \ref{thm:linear-main} and
  then Theorem \ref{thm:main} in Section
  \ref{sec:cimaps} below.

\begin{theorem}
  \label{thm:structureI}
   Let $L:\CC^{g}\to\Mdd$ be a \nonsingular homogeneous linear pencil. Then
   there is a concrete injective envelope $E\ss\OS{d}{d'}$ of $\OB{L}$
   and unitaries $V,W$ such that
  \begin{equation*}
   E=    \{ W^* \begin{pmatrix} x & 0 \\ 0 & \phi(x)\end{pmatrix}V  :
         x \in \oplus_1^N  \OS{d_j}{d_j^\prime} \}
  \end{equation*}
   for some choice of integers $(d_j,d_j^\prime)$ and a
   completely contractive mapping $\phi:\oplus_1^N \OS{d_j}{d_j^\prime} \to \OS{s}{s^\prime}$.
   $($Here $s,s'\in\NN_0$ are such that
     $s+\sum_1^N  d_j =d$ and $s'+\sum_1^N d_j^\prime = d^\prime.)$
 \end{theorem}

 \begin{proof}
  Let $E$ denote an injective envelope of $\mathcal R_L$ and let
  $\mathcal E$ denote an injective envelope of the operator system
  $S_E$.  
   From Proposition \ref{prop:from-system-to-space},
  \begin{equation*}
     \cE = \{ \begin{pmatrix} \zeta  & b^* \\ a & \xi \end{pmatrix} :
      \zeta \in \mathcal A, \xi \in \mathcal B,   a,b \in E\},
  \end{equation*}
   for some unital subspaces $\mathcal A\subseteq\CC^{d\times d}$ and
   $\mathcal B\subseteq \CC^{d^\prime \times d^\prime}$.

   On the other hand, from Corollary \ref{cor:structureIS}, the injective operator system
   $\mathcal E$ has, up to unitary equivalence, the form,
  \begin{equation*}
     \{ \begin{pmatrix} x & 0 \\ 0 & \Phi(x) \end{pmatrix}:
          x\in\oplus \OS{n_j}{n_j} \},
  \end{equation*}
   where $\Phi:\oplus \OS{n_j}{n_j}\to \OS{m}{m}$ is a unital completely contractive
   mapping.
   With respect to $\OS{n_j}{n_j}$
      acting on $\oplus \CC^{n_j}$,
         let $\cQ_j$ denote the $j$-th coordinate $\CC^{n_j}$.

   Thus there is a unitary (block) matrix $U=\begin{pmatrix} u_{ij}\end{pmatrix}$
   such that
  \begin{equation}
   \label{eq:I0}
       \begin{pmatrix} u_{11} & u_{12} \\ u_{21} & u_{22} \end{pmatrix}
       \begin{pmatrix} \zeta & b^* \\ a & \xi \end{pmatrix}
        \begin{pmatrix} u_{11}^* & u_{21}^* \\ u_{12}^* & u_{22}^* \end{pmatrix}
       = \begin{pmatrix} x & 0 \\ 0 & \Phi(x) \end{pmatrix}.
  \end{equation}
    In particular,
  \begin{equation}
   \label{eq:I1}
   \begin{split}
    u_{11}\zeta u_{11}^* +  u_{12}\xi u_{12}^* + u_{12}au_{11}^* + u_{11}b^* u_{12}^* &=
         x \in\oplus \OS{n_j}{n_j}\\
    u_{11} \zeta u_{21}^* +u_{12} \xi u_{22}^* + u_{12}au_{21}^* + u_{11}b^* u_{22}^* &=0 \\
      u_{21} \zeta u_{21}^*  +u_{22}\xi u_{22}^* + u_{22}au_{21}^* + u_{21}b^* u_{22}^* &=\Phi(x).
   \end{split}
  \end{equation}

   Choosing $\zeta = 1$, $\xi=0$ and $a=b=0$ in equation \eqref{eq:I1} gives,
 \begin{equation}
  \label{eq:I2}
     u_{11}u_{21}^* =   0\quad \text{and}\quad      u_{21}u_{11}^* = 0.
 \end{equation}
   Further since in this case the right-hand side of equation \eqref{eq:I0}
   is a projection, it also follows that $u_{11}u_{11}^*$ and $u_{21}u_{21}^*$
   are both projections. Equivalently, $u_{11}$ and $u_{21}$ are partial
   isometries.

  Choosing $\xi=1$, $\zeta=0$ and $a=b=0$ in equation \eqref{eq:I1} gives,
 \begin{equation}
  \label{eq:I2'}
     u_{12}u_{22}^* =   0\quad \text{and}\quad
     u_{22}u_{12}^* =  0.
 \end{equation}
   Moreover, $u_{12}u_{12}^*$ and $u_{22}u_{22}^*$ are projections
   and $u_{12}$ and $u_{22}$ are partial isometries.

   Next, choosing $\zeta=1$, $\xi=1$ and $a=b=0$
   (or using that $U$ is unitary)
   it follows that
 \begin{equation}
  \label{eq:I3}
      u_{11}u_{11}^* + u_{12}u_{12}^* = I \quad\text{and}\quad
      u_{21}u_{21}^* + u_{22}u_{22}^* = I.
 \end{equation}
   Using the fact that all the entries of $U$
   are partial isometries, it now follows that
   $u_{11}u_{11}^*$ and $u_{12}u_{12}^*$ are orthogonal
   projections and $u_{11}^* u_{12}=0$.

   Let
   $\mathcal L \subset \oplus \mathbb C^{n_j}$
      denote the range of $u_{11}$
   so that, by the above relations,
  the  range of $u_{12}$ is $\mathcal L^\perp =(\oplus \mathbb C^{n_j})\ominus \mathcal L.$
   Similarly,
  let $\mathcal K\subset \mathbb C^m$ denote the range of $u_{21}$ so
  that, in view of the above relations, the range
  of $u_{22}$ is $\mathcal K^\perp =\mathbb C^m \ominus \mathcal K$.

 With these notations,
  we have
 \begin{equation*}
  \begin{split}
    W&=  \begin{pmatrix} P_{\cL^\perp} u_{12} \\ P_{\cK^\perp}
      u_{22} \end{pmatrix} :
     \mathbb C^{d^\prime} \to \mathcal L^\perp \oplus \mathcal K^\perp \\
    V&= \begin{pmatrix} P_{\cL} u_{11} \\ P_{\cK} u_{21} \end{pmatrix}
        :\CC^d\to  \mathcal L\oplus \mathcal K
  \end{split}
 \end{equation*}
    are unitaries as is verified by computing $V^*V$
   and $W^*W$
   and noting that each is the identity (on the appropriate space).

  We now turn to proving that $W$ and $V$ satisfy the conclusion of
  the theorem.
  For future reference, observe,
 \begin{equation*}
    WaV^* = \begin{pmatrix} P_{\cL^\perp}u_{12}au_{11}^* P_{\mathcal L} &
            P_{\cL^\perp} u_{12}au_{21}^*P_{\mathcal K} \\
             P_{\cK^\perp} u_{22}au_{11}^* P_{\cL}
               & P_{\cK^\perp} u_{22}au_{21}^* P_{\cK} \end{pmatrix}.
 \end{equation*}
  In view of the second equality in equation \eqref{eq:I1} 
  (choose $\zeta,\xi$ and $b$ equal $0$ to 
  deduce the $(1,2)$ term is $0$ and let
  $\zeta,\xi,a$ be $0$ to deduce the $(2,1)$ term is $0$)
  the off-diagonal
  terms above are $0$.
  From the first and third equalities in equation \eqref{eq:I1},
  there is an $x\in\oplus \OS{n_j}{n_j}$ such that
  $u_{12}au_{11}^* =x$ and $u_{22}au_{21}^*=\Phi(x)$ (again choose
   $\zeta,\xi$ and $b$ equal $0$).  
  Thus, for each $a\in E$ there is an $x\in\oplus \OS{n_j}{n_j}$
  such that
 \begin{equation}
  \label{eq:I4}
     WaV^*  =  \begin{pmatrix} P_{\cL^\perp} x  P_{\mathcal L} &
            0  \\ 0
               & P_{\cK^\perp} \Phi(x) P_{\cK} \end{pmatrix}.
 \end{equation}
  Thus $WaV^*$ has a certain amount of block diagonal structure.

   We now turn to  proving  that the upper left-hand corner of $WaV^*$ has
   the additional block diagonal structure   claimed in the
   theorem; i.e.,  that $ P_{\cL^\perp} x  P_{\mathcal L}\in
   \oplus \OS{d_j}{d_j^\prime}$
   for some choice of $d_j$ and $d_j^\prime$.
   Observe that $P_{\cL}=u_{11}u_{11}^*$, the projection onto $\mathcal L$
   is contained in $\oplus \OS{n_j}{n_j}$
   (choose $\zeta =1$ and $a,b,\xi$ equal $0$ in \eqref{eq:I1}).
   Hence, with respect to this decomposition,
  \begin{equation*}
    P_{\cL}= \begin{pmatrix} P_1 & 0 & \dots & 0 \\ 0 & P_2 & \dots & 0 \\ \vdots & \vdots & \dots & \vdots\\
             0&0& \dots & P_N \end{pmatrix}.
  \end{equation*}
   It follows that each $P_j$ is a projection which commutes
   with the projection $Q_j$
   onto $\mathcal Q_j$ equal to the $\mathbb C^{n_j}$
   summand of $\oplus \mathbb C^{n_k}$. Letting $\mathcal L_j$ denote the range of
   $P_j$ it follows
   that $$\mathcal L=\oplus \mathcal L_j.$$
   Similarly, $P_{\cL^\perp}=u_{12}u_{12}^*$
   is in $\oplus \OS{n_j}{n_j}$ and thus commutes with each $Q_j$. Consequently,
   $$\mathcal L^\perp =  \oplus \mathcal L_j^\perp,$$
   where $\mathcal L_j^\perp$ is the orthogonal complement of $\cL_j$ in
   $\mathbb C^{n_j}=\mathcal Q_j$.

   Let $d_j$ and $d_j^\prime$ denote the dimensions of $\mathcal L_j$ and
   $\mathcal L^\perp_j$ respectively (thus $d_j+d_j^\prime = n_j)$.
   Identifying $\mathcal L_j$ with $\mathbb C^{d_j}$ and
   $\mathcal L_j^\perp$ with $\mathbb C^{d_j^\prime}$,
   it follows, for $x\in\oplus \OS{n_j}{n_j}$, that
 \begin{equation}
  \label{eq:I5}
     P_{\cL^\perp} x P_{\cL} \in \oplus \OS{d_j}{d_j^\prime}.
 \end{equation}

As explained above \eqref{eq:I4},
for every $a\in E$
we have
\beq\label{eq:I4b}
a = W^*
\begin{pmatrix}
 P_{\cL^\perp} x  P_{\mathcal L} &
            0  \\ 0
               & P_{\cK^\perp} \Phi(x) P_{\cK} \end{pmatrix} V
=
W^*
\begin{pmatrix}
x & 0\\
0&  P_{\cK^\perp} \Phi(x) P_{\cK} \end{pmatrix} V
\in W^* (\oplus \OS{d_j}{d_j^\prime} \oplus \OS{s}{s'})V.
\eeq
 Here $x$ denotes   $u_{12}au_{11}^*$
and  $s=\dim\cK$ , $s'=\dim\cK'$.
The second equality uses
that $\cL$ is the range of $u_{11}$ and
$\cL^\perp$ is the range of $u_{12}$.
In this way, since $\oplus\OS{n_j}{n_j}$ is the set of
all linear maps $\cL\oplus\cL^\perp\to\cL\oplus\cL^\perp$,
we identify $x\in\oplus \OS{d_j}{d_j^\prime}$ with
$\begin{pmatrix} 0&0\\ x&0\end{pmatrix}:
\cL\oplus\cL^\perp\to\cL\oplus\cL^\perp$.
Defining
$$
\phi:\oplus \OS{d_j}{d_j^\prime}\to\OS{s}{s'}, \quad
x\mapsto  P_{\cK^\perp} \Phi(x) P_{\cK},
$$
we thus obtain $E \subset F$ with $F$ defined by
\beq\label{eq:Fdef}
F= \{ W^*\begin{pmatrix} x&0\\
0&\phi(x)\end{pmatrix}V: x\in\oplus \OS{d_j}{d_j^\prime}\} .
\eeq
Note $\phi$ is completely contractive.

To prove
the reverse inclusion, $E \supseteq F$, pick any $f$ in $F$.
This $f$ corresponds to  an  $x\in\oplus \OS{d_j}{d_j^\prime}$
in the definition \eqref{eq:Fdef} of $F$.
By equation \eqref{eq:I0},
   there exist $\zeta,\xi,a,b$ such that
\[
  U \begin{pmatrix} \zeta & b^* \\ a &\xi \end{pmatrix} U^* 
    =\begin{pmatrix} x & 0 \\ 0 & \Phi(x) \end{pmatrix}.
\]
 Hence,
\[
  \begin{pmatrix} 0&0\\ a& 0 \end{pmatrix} =
  \begin{pmatrix} 0& 0 \\ u_{12}^* & u_{22}^* \end{pmatrix}
  \begin{pmatrix} x & 0 \\ 0 & \Phi(x) \end{pmatrix}
  \begin{pmatrix} u_{11} & 0 \\ u_{21} & 0 \end{pmatrix} 
  = \begin{pmatrix} 0 & 0 \\ u_{12}^* x u_{11} + u_{22}^* \Phi(x) u_{21}&0
        \end{pmatrix}.
\]
  We conclude 
\[
  a= u_{12}^* x u_{11} + u_{22}^* \Phi(x) u_{21}.
\]
 Multiplying this identity on the left by $u_{12}$ and the right
  by $u_{11}^*$ and using various orthogonality relations gives
\[
  u_{12}au_{11}^* = P_{\cL^\perp} x P_{\cL}=x.
\]
   There is a $y$ such that
\[
 \begin{pmatrix} y & 0\\ 0 & \Phi(y)\end{pmatrix}
   =    U \begin{pmatrix} 0 & 0 \\ a & 0 \end{pmatrix} U^* \\
   = WaV^* \\
   = \begin{pmatrix} u_{12}au_{11}^* & \star \\ \star & u_{22}a u_{21}^* \end{pmatrix}.
\]
  It follows that $y=x$ and thus 
   $\Phi(x)=u_{22}au_{21}^* =P_{\cK^\perp}\Phi(x)P_{\cK} =\phi(x)$.
  In conclusion,
\[
   f = W^* \begin{pmatrix} x & 0 \\ 0 & \phi(x)\end{pmatrix} V =a\in E,
\]
  for an $x\in \oplus \OS{d_j}{d_j^\prime}$.
\end{proof}

We point out that the $L$ in Theorem \ref{thm:structureI} plays 
no role beyond defining a subspace $Z$ of
rectangular
matrices. The theorem is a characterization of injective envelopes of 
a space $Z$ of rectangular matrices, or equivalently of complete isometries from a direct sum of rectangular matrices into $Z$.

\section{The Main Results}
 \label{sec:cimaps}

  In this section we prove our main result on completely isometric
  maps  $\OB{L}\to \OS{\ell}{\ell^\prime}$, Theorem \ref{thm:linear-main},
  and then Theorem \ref{thm:main}.

\subsection{Completely isometric maps $\OB{L} \to \OS{\ell}{\ell^\prime}$}

\begin{remark}\rm
Given a completely isometric
$\Phi:\OB{L}\to \OS{\ell}{\ell^\prime}$,
   it is tempting to guess that, up
   to unitaries,
  \begin{equation*}
   \Phi(L(x))= \begin{pmatrix} L(x) & 0 \\ 0 & \phi(L(x)) \end{pmatrix},
  \end{equation*}
   for some completely contractive $\phi$.
   However, if $L(x)= x\oplus \frac12 x$, then the mapping
   $\Phi(L(x))=x$ is completely isometric, but not of
   the form above.  This prompts us to decompose
  $L=\tL\oplus J$ in the theorem below.
 \end{remark}

 We now rephrase Theorem \ref{thm:linear-main}
 (together with Lemma \ref{lem:minDim}).
 Then we set about to prove it. Theorem \ref{thm:linear-main} will follow
 as soon as the equivalence between ``minimal pencil'' and ``minimal dimensional
 defining pencil for a ball'' is established in Corollary \ref{cor:minDim}.

\begin{theorem}\label{thm:bigDeal}
     Given a \nonsingular homogeneous linear pencil $L:\CC^{ g}\to\Mdd$, there are 
     homogeneous linear pencils
     $\tL$ and $J$ and unitaries $Q$ and $G$ such that
 \begin{enumerate}[\rm (1)]
  \item $L= Q(\tL \oplus J)G^*;$
  \item $\|L(X)\|=\|\tL(X)\|$ for all $n\in\NN$ and $X\in (\CC^{n\times n})^{g}$, i.e., $\tL$ is equivalent to $L$;
\item
there are $d_j,d_j'\in\NN$ such that
$\tL:\CC^\tg\to\oplus_{1}^N \CC^{d_j'\times d_j}$
and
the injective envelope of $\OB{\tL}$ is $\oplus_{1}^N \OS{d_j}{d_j^\prime}$.
 \end{enumerate}
If  $\Phi:\OB{L}\to \OS{\ell}{\ell^\prime}$
     is completely isometric, then
   \begin{equation*}
   \Phi(L(x))= U \begin{pmatrix} \tL(x) & 0 \\ 0 & \phi(\tL(x)) \end{pmatrix} V^*,
  \end{equation*}
    for some completely contractive
    $\phi:\oplus\OS{d_j}{d_j^\prime}\to\OS{\ell-\sum d_j}{\ell'-\sum g'_j} $ and unitaries $U,V$.
\end{theorem}

  \begin{proof}
     Let $E\ss\OS{d}{d'}$ denote a concrete injective envelope of $\OB{L}$. By
     Theorem \ref{thm:structureI} there exist unitaries $Q,G$ such that
 $$
  E=\{Q\begin{pmatrix} a & 0 \\ 0 & \psi(a) \end{pmatrix}G^* : a\in \oplus_1^N \OS{d_j}{d_j^\prime} \}
 $$
 for some choice of integers $(d_j,d_j^\prime)$ and a
    completely contractive mapping $\psi:\oplus_1^N \OS{d_j}{d_j^\prime} \to \OS{s}{s^\prime}$, where $s,s'\in\NN_0$ with
    $s+\sum_1^N  d_j =d$ and $s'+\sum_1^N d_j^\prime = d^\prime.$

   The mapping $\OB L\to \oplus \OS{d_j}{d_j^\prime}$
defined on $L(x)\in\OB L$ by
\begin{equation*}
    L(x)=Q \begin{pmatrix} a & 0 \\ 0 & \psi(a) \end{pmatrix}G^* \mapsto a
 \end{equation*}
   is linear and completely isometric,
   and so $x\mapsto L(x)\mapsto a$
   defines a linear map $\tL:\CC^{g}\to \oplus \OS{d_j}{d_j^\prime}$ satisfying
   $\|L(X)\|=\|\tL(X)\|$ for all $X$.
   Similarly, $J:\CC^{g}\to \oplus \OS{d_j}{d_j^\prime}$ is constructed by
   mapping $x\mapsto L(x)\mapsto \psi(a)$. By construction,
 \begin{equation*}
   L(x) =Q \begin{pmatrix} \tL(x) & 0 \\ 0 & J(x) \end{pmatrix} G^*.
 \end{equation*}

   Now if $\Phi: \OB{L} \to \OS{\ell}{\ell^\prime}$ is completely isometric, then,
   $\tau: \OB{\tL}\to \OS{\ell}{\ell^\prime}$ given by $\tau(\tL(x))=\Phi(L(x))$
   is well defined and completely isometric. Consider the following commutative
   diagram:
\[
\xymatrix{
&   \OB{\tL} \ar@{^{(}->}[d]_{\text{c.isom.}}   \ar[drr]^{\tau}
\ar@{^{(}->}@/^-5pc/[ddd]   \\
&   \OB L\ar@{^{(}->}[d]_{\text{inj.env.}} \ar[rr]^{\Phi} && \OS {\ell}{\ell'} \\
&   E \ar@{{}{--}>}[urr]_{\tilde\Phi}               \\
&   \oplus\OS{d_j}{d_j^\prime}\ar@{^{(}->}[u]_{\text{c.isom.}}
\ar@/^-1pc/[uurr]_{\bar\Phi}
  }
\]

Since $E$ is the injective envelope of $\OB L$, $\Phi$ extends to
a completely isometric $\tilde\Phi:E\to\OS {\ell}{\ell'}$ (cf.~Corollary
\ref{cor:essential} and Remark \ref{rem:ess}).
Hence $\bar\Phi:\oplus\OS{d_j}{d_j^\prime}\to\OS {\ell}{\ell'}$, being
the composite of two completely isometric maps, is completely isometric.
Thus by Theorem \ref{thm:structureCI},
there are unitaries $U, V$ such that
$$
\bar\Phi(a)=U \begin{pmatrix} a & 0\\
0& \phi(a)
\end{pmatrix}V^*$$
for some completely contractive $\phi$ and all $a\in\oplus\OS{d_j}{d_j^\prime}$.
 In particular, this holds for all $a\in\OB {\tL}$, that is,
$$
\Phi(L(x))=\tau(\tL(x))=U \begin{pmatrix} \tL(x) & 0\\
0& \phi(\tL(x))
\end{pmatrix}V^*,
$$
finishing the proof.
(Note: along the way we have shown that
$\oplus\OS{d_j}{d_j^\prime}$ is the injective envelope
of $\OB {\tL}$.)
  \end{proof}

A \nonsingular linear pencil $L$ is called \textbf{minimal} if
the concrete injective envelope of $\OB{L}$ is
$\oplus \OS{d_i}{d_i^\prime}$ for some $d_i,d_i'$.
From Theorem \ref{thm:bigDeal} it follows
that every homogeneous linear \nonsingular pencil is equivalent to a minimal one.
Recall the notion of 
a minimal dimensional defining pencil for a ball from Section \ref{sec:main}.

\begin{corollary}\label{cor:minDim}
A homogeneous linear pencil $L$ is
minimal if and only if it is a
minimal dimensional defining pencil for a ball.
\end{corollary}

\begin{proof}
Suppose $\cB_{L'}=\cB_L$ for some $d'\times d$ homogeneous linear
\nonsingular pencil $L'$. Then
$$\Psi:\cR_L\to\OS{d}{d'},\quad
L(x)\mapsto L'(x)$$ is completely isometric.
Since $L$ is minimal, the injective envelope of $\cR_L$ is $\oplus \OS{d_j}{d_j'}$.
Hence $\Psi$ extends to a completely isometric $\Psi:\oplus \OS{d_j}{d_j'}
\to\OS{d}{d'}$ and is thus described by Theorem \ref{thm:structureCI}. There
exist unitaries $U,V$ and a completely contractive $\psi$ such that
$$
\Psi(y)=U \begin{pmatrix} y & 0 \\ 0 & \psi(y) \end{pmatrix} V^*
$$
for $y\in\oplus \OS{d_j}{d_j'}$. Applying this to $y=L(x)$ yields
$$
L'(x)=\Psi(L(x))= U \begin{pmatrix} L(x) & 0 \\ 0 & \psi(L(x)) \end{pmatrix} V^*.
$$
Thus if $L'$ is a minimal dimensional defining pencil for a ball,
$\psi=0$ and $L'(x)=U L(x) V^*$, so $L$ is a minimal dimensional defining pencil for a ball, too.

Conversely, if $L$ is a minimal dimensional defining pencil for a ball, then
by Theorem \ref{thm:bigDeal}, $\cB_L=\cB_{\tL}$ and the size of $\tL$ is
at most that of $L$. By the minimality of $L$, $J=0$ and
so $L$ equals the minimal pencil $\tL$ up to unitaries. Hence
$L$ is minimal.
\end{proof}

The proof shows that a minimal dimensional defining pencil for a ball
is also minimal with respect to $\ell$ and $\ell'$, respectively.

\subsection{Pencil Ball Maps}
In this subsection we present the proof of our main result, Theorem \ref{thm:main}.

\begin{lemma}\label{lem:isoDerivative}
Suppose $L$ is a minimal linear pencil and
let $f:\cB_{L}\to\OS{\ell}{\ell'}$ be a pencil ball map with $f(0)=0$.
Then
$$
f^{(1)}(x)=U
\begin{pmatrix}L(x)&0\\0&\ \phi(L(x))
\end{pmatrix} V^*
$$
for some completely contractive $\phi$ and unitaries $U,V$.
\end{lemma}

\begin{proof}
By definition,
\beq\label{eq:bind}
\|L(X)\|< 1 \; \Rightarrow \; \|f(X)\|\leq 1, \quad
\|L(X)\|= 1 \; \Rightarrow \;\|f((\exp(it)X)\|= 1\text{ for a.e. }t\in\RR.
\eeq
For $X\in\cB_{L}$, $\begin{pmatrix}0&X\\0&0\end{pmatrix} \in \cB_{L}$
so
$$
f(\begin{pmatrix}0&X\\0&0\end{pmatrix} )= \begin{pmatrix}0&f^{(1)}(X)\\
0&0\end{pmatrix}
$$
has norm at most $1$ and is of norm $1$ (a.e.) if $X\in\bBL$.
By linearity, \eqref{eq:bind} implies $\|f^{(1)}(X)\|=\|L(X)\|$
for all $X\in\cB_L$.

Since $L$ is minimal,
$f$ induces
$$h:\OB{L}\to\OS{\ell}{\ell'}, \quad h(L(x))=f(x).
$$
Moreover, $h^{(1)}(L(x))=f^{(1)}(x)$ is a complete
isometry $\OB{L}\to\OS{\ell}{\ell'}$. Theorem \ref{thm:bigDeal}
implies
$$
f^{(1)}(x)=
U
\begin{pmatrix}L(x)&0\\0&\ \phi(L(x))
\end{pmatrix} V^*
$$
for unitaries $U,V$ and a completely contractive $\phi$.
\end{proof}

Before moving on to the general situation we explain the main idea
for the case of the
quadratic homogeneous component of a pencil ball map.

\begin{lemma}\label{lem:2}
Suppose $L$ is a minimal linear pencil and
let $f:\BL \to \OS{\ell}{\ell^\prime}$ be a pencil ball map with $f(0)=0$.
Suppose $f^{(1)}$ is as in Lemma {\rm \ref{lem:isoDerivative}}.
Then $f^{(2)}$ has the form
$$
f^{(2)}(x)= U
\begin{pmatrix}0&0\\0& \star
\end{pmatrix} V^*.
$$
\end{lemma}

\begin{proof}
  Write the homogeneous linear pencil $L$ as
 \begin{equation*}
    L(x)=\sum_{j=1}^{g} A_j x_j.
 \end{equation*}

  Given a tuple $X\in  \cB_L(n)$,
  define,
 \begin{equation*}
      T_1 = \begin{pmatrix} 0 & \lambda I & 0 \\ 0 & 0 & X_1 \\
             0 & 0 & 0 \end{pmatrix}, \quad
      T_j = \begin{pmatrix} 0 & 0 & 0\\ 0 & 0 & X_j\\
             0 & 0 & 0\end{pmatrix} \quad \text{for }j\geq 2,
 \end{equation*}
    where $\lambda$ will be chosen later.
    For now, it suffices to  note
 \begin{equation*}
    L(T) = \begin{pmatrix} 0 & A_1\otimes \lambda I & 0 \\
                    0 & 0 & L(X) \\
                  0 & 0 & 0 \end{pmatrix},
 \end{equation*}
  and thus $\|L(T)\|\le 1$ if $|\lambda|$ is sufficiently small.
  For use below, observe that
 \begin{equation*}
    T_1 T_j = \begin{pmatrix} 0 & 0 & \lambda X_j \\
                0&0&0\\0&0&0\end{pmatrix}, \quad
    T_k T_j =  0  \quad \text{for }k\geq 2.
 \end{equation*}

  Write $f^{(2)}$, the quadratic part of $f$ as
 \begin{equation*}
   f^{(2)} = \sum_{a,b=1}^{g} f_{a,b}x_a x_b.
 \end{equation*}
  With this notation,
 \begin{equation*}
   f^{(2)}(T)=\begin{pmatrix} 0 & 0 & \sum f_{1,j}\otimes \lambda X_j\\
              0&0&0\\0&0&0\end{pmatrix}
 \end{equation*}
and
 \begin{equation*}
  f(T) =\begin{pmatrix} 0 & f^{(1)}(\lambda I,0,\ldots,0) & \sum f_{1,j}\otimes \lambda X_j\\
            0 & 0 & f^{(1)}(X) \\ 0 & 0 & 0 \end{pmatrix}.
 \end{equation*}
  And further, for $\lambda$ of sufficiently small norm, $f(T)$ is
  a contraction and $\|L(X)\|=1$ implies $\|f(T)\|=1$. Hence,
 \begin{equation*}
  \Lambda(X) = \begin{pmatrix}  \sum f_{1,j}\otimes \lambda X_j\\
            f^{(1)}(X) \end{pmatrix}
 \end{equation*}
 for small enough $\lambda$ satisfies
$$1 = \|L(X)\| = \|f(T)\| \geq \|\Lambda(X)\| \geq \|f^{(1)}(X)\|=
\|L(X)\|.$$
Hence by linearity, $\|\Lambda(X)\| = \|f^{(1)}(X)\|=
\|L(X)\|$ for all $X$.

  Since $L$ is nondegenerate,
 \begin{equation*}
  \Delta: \OB{L}\to \OS{\ell}{2\ell^\prime},\quad L(x) \mapsto \Lambda(x)
 \end{equation*}
 is a well defined linear map.
 Also,
 \begin{equation*}
  \Delta(L(x)) =\begin{pmatrix} J(L(x)) \\ f^{(1)}(x) \end{pmatrix}
 \end{equation*}
  for some linear $J:\OB{L}\to \OS{\ell}{\ell^\prime}$. In the coordinates
  given by $U,V$:
 $$
\Delta(L(x)) =\begin{pmatrix} J_{11}(L(x)) & J_{12}(L(x))\\
J_{21}(L(x)) & J_{22}(L(x)) \\
L(x)&0\\0&\ \phi(L(x))
\end{pmatrix},
 $$
where $\phi$ is a complete contraction.

  Obviously, $\Delta : \OB{L}\to \OS{\ell}{2\ell^\prime}$
  is completely isometric. It thus has a completely
  isometric extension to $\oplus \OS{d_j}{d_j^\prime}$.
  Moreover, since the $(3,1)$ term, $L(x)$, is the identity
  on $\OB L$, the only extension of this term to
  all of $\oplus \OS{d_j}{d_j^\prime}$ is the identity (cf.~Lemma
  \ref{lem:F-map-auto-identity}). It
  now follows that the extension of $J_{j1}$ must be
  zero (cf.~Theorem \ref{thm:bigDeal}).
   Now repeat the argument with $k$ replacing $1$ to
    get 
 \[ f_{j,k}=U
    \begin{pmatrix}0& f_{j,k}^{1,2}\\0&  f_{j,k}^{2,2}
    \end{pmatrix} V^*
 \]
 for all $j,k$.

    To show that $f_{j,k}^{1,2}=0$ for all $j,k$ we simply repeat
the entire argument using
      \[
      T_1 = \begin{pmatrix} 0 & 0 & 0 \\ \lambda I & 0 & 0 \\
             0 & X_1 & 0 \end{pmatrix}, \quad
      T_j = \begin{pmatrix} 0 & 0 & 0\\ 0 & 0 & 0\\
             0 & X_j & 0\end{pmatrix} \quad \text{for }j\geq 2.
             \qedhere \]
 \end{proof}

\begin{proof}[Proof of Theorem {\rm \ref{thm:main}}]
We may replace $L$ by a minimal pencil equivalent to it
and thus assume that $L$ is minimal.
Also, by Lemma \ref{lem:isoDerivative},
$$f^{(1)}(x)=U
\begin{pmatrix}L(x)&0\\0&\ \phi(L(x))
\end{pmatrix} V^*
$$
for unitaries $U,V$ and a complete contraction $\phi$. We will
prove that for $m\geq 2$,
\beq\label{eq:induct}
f^{(m)}(x)=U
\begin{pmatrix}0&0\\0&\ \phi^{(m)}(L(x))
\end{pmatrix} V^*.
\eeq
This holds for $m=2$ by Lemma \ref{lem:2}. Let $m\geq 2$, assume
\eqref{eq:induct} holds up to $m-1$ and write
$$f^{(m)}=\sum_{\substack{\; w\in\langle x\rangle\\ |w|=m}} f_w w, \quad
f_w\in\CC^{\ell'\times \ell}.$$

Fix $i_1,\ldots,i_{m-1}\in\{1,\ldots,\tg\}$ and consider
the coefficient $f_{x_{i_1}\cdots x_{i_{m-1}} x_j}$ of $f$.
To ``isolate'' this coefficient we
construct  block $(m+1)\times (m+1)$ matrices $T_1,\ldots, T_{g}$ as follows.
Each $T_i$ has $X_i$ as its $(m,m+1)$ entry. In addition to that, we put
a $\lambda I$ as the $(k,k+1)$ entry of $T_{i_k}$.
All the other entries not explicitly given above are set to $0$.
(Note that $T_i$ might have several nonzero superdiagonal entries
as the $i_k$ are not necessarily pairwise distinct.)

For example, if $m=4$ and $x_{i_1}x_{i_2}x_{i_3}=x_1x_2x_1$, then
$$
T_2=\begin{pmatrix}
0&0&0&0&0\\
0&0&\la I&0&0\\
0&0&0&0&0\\
0&0&0&0&X_2 \\
0&0&0&0&0
\end{pmatrix},$$
but $i_1=i_3=1$, so $\lambda I$ occurs twice in $T_1$, once as the
$(1,2)$ entry and also as the $(3,4)$ entry:
$$
T_1=\begin{pmatrix}
0&\la I&0&0&0\\
0&0&0&0&0\\
0&0&0&\la I&0\\
0&0&0&0&X_1\\
0&0&0&0&0
\end{pmatrix}.
$$
Thus
$$
T_1T_2T_1=
\begin{pmatrix}
0&0&0&\la^3 I&0\\
0&0&0&0&0\\
0&0&0&0&0\\
0&0&0&0&0\\
0&0&0&0&0
\end{pmatrix}.
$$
More generally,
if $u\in\langle x\rangle$ is of degree $k$, then all
the nonzero entries of $u(T)$ are on the $k$-th superdiagonal.

Let us consider a product $T_{j_1}\cdots T_{j_{m-1}}$ of the
$T_j$'s of length $m-1$. Its
$(1,m)$ entry is
$$
(T_{j_1})_{1,2} (T_{j_2})_{2,3}\cdots (T_{j_{m-1}})_{m-1,m}
$$
and is nonzero if and only if $(T_{j_k})_{k,k+1}\neq 0$ for all
$k=1,\ldots,m-1$. So
the only product of the $T_j$'s of length $m-1$ with a nonzero entry
in the $m$-th column is
$$
T_{i_1}\cdots T_{i_{m-1}}= \begin{pmatrix}
0&\cdots&0&\lambda^{m-1}I&0 \\
\vdots&\ddots &&0&\star \\
\vdots &&\ddots&&0\\
\vdots&&&\ddots&\vdots\\
0&\cdots&\cdots&\cdots&0
\end{pmatrix}.
$$
Likewise the only words that produce a nonzero top right entry
are those of length $m$ that start with
$x_{i_1}\cdots x_{i_{m-1}}$,
e.g.~$x_{i_1}\cdots x_{i_{m-1}} x_j$ produces
$\lambda^{m-1} T_j$ in the top right corner.
Hence the top right entry in $f(T)$ is
$$
\delta=\sum_j f_{x_{i_1}\cdots x_{i_{m-1}}x_j} \otimes \lambda^{m-1} X_j,
$$
so
$$
f(T)= \begin{pmatrix}
0&\star& \cdots & \cdots&
\delta
\\
0 & 0 & \star & \cdots & \star \\
\vdots & \vdots & \ddots & \ddots & \vdots\\
\vdots & \vdots & \vdots & \ddots & f^{(1)}(X)\\
0 & 0 & \cdots & \cdots & 0
\end{pmatrix}.
$$
Here all the entries denoted by
$\star$ come from homogeneous components of $f$ of degree $<m$ and
are of degree $\leq 1$ in the $X_i$.
The entries in the last column are all of degree $=1$ in the $X_i$.
Since $L$ is nondegenerate, $\Delta:\OB L\to \OS{\ell}{m\ell'}$ which
maps $L(X)$ to the last column of $f(T)$ (without the bottom entry $0$)
is a well defined linear
map. Furthermore, $\delta=J(L(X))$ for some linear
$J:\OB L\to\OS{\ell}{\ell'}$.

Assume $X\in\BL$ and $|\lambda|$ is small. Then all the entries $\star$
have norm $<1$ and $\|f(T)\|\leq 1$. Furthermore, $\|L(X)\|=1$ implies
$\|f(T)\|=1$. So
$$1 = \|L(X)\| = \|f(T)\| \geq \|\Delta(L(X))\| \geq \|f^{(1)}(X)\|=
\|L(X)\|.$$
Hence by linearity, $\|\Delta(L(X))\| = \|f^{(1)}(X)\|=
\|L(X)\|$ for all $X$.

The last column of
$f(T)$ in the coordinates given by $U,V$ is
$$
\Delta(L(x))=
\begin{pmatrix}
J_{11}(L(x))&J_{12}(L(x))\\
J_{21}(L(x))&J_{22}(L(x))\\
0 & 0 \\
0 & \star \\
0 & 0 \\
0 & \star \\
\vdots & \vdots \\
L(x) & 0\\
0 & \phi(L(x))
\end{pmatrix},
$$
where $\phi$ is completely contractive.
(Here we have used the induction hypothesis \eqref{eq:induct}
on all the terms of degree between 2 and $n-1$.)
We can now proceed as in the proof of Lemma \ref{lem:2} to conclude
that $J_{j1}(L(x))=0$. Similarly one obtains $J_{12}(L(x))=0$.
All this proves \eqref{eq:induct} and finishes the proof.
\end{proof}

\begin{proof}[Alternative proof of Theorem {\rm \ref{thm:main}}]
We may replace $L$ by the minimal pencil $\tilde L$ equivalent to it
and thus assume that $L$ is minimal.
Also, by Lemma \ref{lem:isoDerivative},
$$f^{(1)}(x)=U
\begin{pmatrix}L(x)&0\\0&\ \phi(L(x))
\end{pmatrix} V^*
$$
for unitaries $U,V$ and a complete contraction $\phi$. We will
prove that for $m\geq 2$,
\beq\label{eq:induc}
f^{(m)}(x)=U
\begin{pmatrix}0&0\\0&\ \phi^{(m)}(L(x))
\end{pmatrix} V^*.
\eeq
This holds for $m=2$ by Lemma \ref{lem:2}.
Consider $f^{(m)}$,
$$f^{(m)}=\sum_{\substack{\; w\in\langle x\rangle\\ |w|=m}} f_w w, \quad
f_w\in\CC^{\ell'\times \ell}.$$
For notational convenience we omit $U, V$ in what follows.

  Fix a word $v\in\langle x\rangle$ of length $m-1$. To show that $f_{vx_k}=0$ for each
  $k$ and hence $f_w=0$ for each word $w$ of length $m$, we shall use
  compressions of the creation operators on Fock space to
  build more elaborate versions of the block matrices in the proof
  of Lemma \ref{lem:2}. To be specific,
let $\mathcal K=\mathcal K_m$ denote the Hilbert space
  with orthonormal basis $\{ u\in\langle x\rangle: 1\le |u|\le m\}$.
  Define $S_j$ on $\mathcal K$ by $S_j u=x_ju$ provided $|u|<m$
  and $S_ju=0$ if $|u|=m$. Note that $S_j^*(x_ju)=u$ and $S_j^*(w)=0$ if
  the word $w\in\mathcal K$ does not begin with $x_j$.

  For $n\in \NN$ define $E_j^{(n)}:\CC^n \to \mathcal K \otimes \CC^n$ by $y \mapsto x_j\otimes y$.

  Given $n$ and a tuple $X \in (\Mnns)^{\tg}$, define, with
  $\lambda>0$ to be chosen later,  matrices
  $T_j$ mapping from $(\mathcal K\otimes \CC^n) \oplus \CC^n \oplus \CC^n$
  to itself
  by
 \begin{equation*}
   T_j =\begin{pmatrix} \lambda S_j \otimes I & \lambda E_j & 0\\
                    0&0&X_j \\ 0&0&0\end{pmatrix}.
 \end{equation*}
 (Here $E_j=E_j^{(n)}$).
  Note that
 \begin{equation*}
    L(T)^* L(T) = \begin{pmatrix}  \star  &\star & 0 \\ \star &\star &0 \\
                       0&0&L(X)^*L(X) \end{pmatrix},
 \end{equation*}
 where all the entries denoted by $\star$ are quadratic in $\lambda$.
  Thus we can choose $\lambda>0$ small enough so that if $\|L(X)\|=1$,
  then $\|L(T)\|=1$ too.

  As an example, let us compute $u(T)$ for $u=x_1x_2x_3$:
  $$
u(T)=T_1T_2T_3=\begin{pmatrix}
\la^3 (S_1S_2S_3\otimes I) & \la^3 (S_1S_2\otimes I)E_3 & \la^2
(S_1\otimes I)E_2X_3 \\
0&0&0\\
0&0&0
\end{pmatrix}
$$
and thus
$$
u(T) ((w\otimes y_1)\oplus y_2\oplus y_3) = \la^3 (S_1S_2S_3 w \otimes y_1) +
\la^3 (S_1S_2 x_2\otimes y_2) + \la^2 (S_1 x_2 \otimes X_3y_3)
$$
for $w\in\cK$ and $y_i\in\CC^n$.
The general calculation along the same lines yields
\begin{equation*}
  f(T)=\begin{pmatrix} \star & \star & \delta \\ 0 &0  & f^{(1)}(X) \\ 0&0&0\end{pmatrix}
  \end{equation*}
for
 \begin{align*}
  \delta &= \sum_k \sum_j \sum_{|u|\le m-2} \lambda^{|u|+1} f_{ux_jx_k} \otimes \big(( u(S)\otimes I ) E_j X_k\big) .
 \end{align*}
By Lemma \ref{lem:isoDerivative}, $\|L(X)\|=\|f^{(1)}(X)\|$ for all
$X$.

  Let $P_v:\mathcal K \otimes \CC^n \to \CC^n$ denote the
  projection $P_v (\sum_u u\otimes y_u) = y_v$.
For $y\in\CC^n$, and $u\in\cK$ with $|u|\leq m-2$, we have
$$P_v(( u(S)\otimes I ) E_j X_k\big) y=
P_v(u(S)x_j\otimes X_k y)= P_v(u x_j\otimes X_k y) =
\begin{cases}
X_k y& \text{if }ux_j=v \\
0 & \text{otherwise.}
\end{cases}
$$
  We extend $P_v$ to the projection
  $\Pi_v:\CC^{d'\times d}\otimes \mathcal K\otimes \CC^n \to \CC^{d'\times d}
  \otimes\CC^n$, $\Pi_v=I\otimes P_v$.
  With this notation,
 \begin{equation*}
      \begin{pmatrix} \Pi_v & 0 & 0\\ 0 & I & 0\end{pmatrix} f(T)
       \begin{pmatrix} 0\\ 0\\ I\end{pmatrix}=
        \begin{pmatrix}  \sum_k \lambda^{m-1} f_{vx_k} \otimes X_k \\ f^{(1)}(X) \end{pmatrix}
        .
 \end{equation*}

 Now proceed as in the proof of Lemma \ref{lem:2}
 to conclude
 \[
 f_{vx_k}(x)=\begin{pmatrix}0&0\\0&\ \phi^{(m)}(L(x))
 \end{pmatrix}
  \]
(in the coordinates given by $U,V$) for some completely contractive
$\phi^{(m)}$.
\end{proof}

\section{More generality}
 \label{sec:more}

In this section we extend the main results presented so far in two directions.
First of all, we use linear fractional transformations to classify
pencil ball maps that do not preserve the origin. For the second generalization
we study pencil ball maps mapping between two pencil balls and
preserving the boundary.

\subsection{Linear fractional transformations}
\label{subsec:linearfrac}
We provide a cursory treatment of linear fractional maps and refer
the reader to \cite[Section 5]{HKMS} for details and proofs.

Let $\Bll:=\{X\in\Mll :\|X\|<1\}$.
For a given $\ell^\prime \times \ell$ scalar matrix $v$ with $\|v\|<1$, define
$\cF_v:\Bll\to \Bll$ by
\beq\label{eq:deffIntro}
 \cF_v (u):=v-(I_{\ell'}-vv^\T )^{1/2}u(I_\ell-v^\T u)^{-1}(I_\ell-v^\T v)^{1/2}.
\eeq
 Of course it must be shown that $\cF_v$ actually takes values
 in $\Bll$; this is done in \cite[Lemma 5.2]{HKMS}.

Linear fractional transformations such
  as $\cF_v$ are common in circuit and system theory,
  since they are associated with energy conserving pieces of a circuit,
  cf.~\cite{Woh}.

Notice that if $\ell=\ell'=1$, then $v$ and $u$ are scalars,
hence
$$ \cF_v (u)= (v-u)(1-u\bar v)^{-1}
= (1-u\bar v)^{-1}(v-u).$$
Now fix $v\in\DD$ and consider the map $\DD\to\CC$,
$
u\mapsto \cF_v (u).
$
This map is a
linear fractional map that maps the unit disc to the unit disc, maps
the unit circle to the unit circle, and maps $v$ to 0.
The geometric interpretation of the map in
(\ref{eq:deffIntro}) is similar: 

\begin{lemma}[Lemma 5.2 in \cite{HKMS}] \label{lemma:multilinfracIntro}
Suppose that  $N\in \mathbb{N}$ and $V\in\Bll(N)$.
\begin{enumerate}[\rm (1)]
\item
$U\mapsto \cF_V(U)$ maps $\Bll(N)$ into itself
with boundary to the boundary.
\item If $U\in\Bll(N)$, then
$\cF_V(\cF_V(U))=U.$
\item $\cF_V(V)=0$ and $\cF_V(0)=V$.
\end{enumerate}
\end{lemma}

\subsection{Classification of pencil ball maps}
 \label{sec:classify}
   General pencil ball maps $f$ -- those where $f(0)$ is not necessarily
   $0$ -- are described using the linear fractional transformation $\cF$.

\begin{corollary}\label{cor:main}
     Suppose $L$ is a \nonsingular homogeneous linear pencil and
     $\tL$ is a minimal dimensional defining pencil for $\cB_L$.
     If $f:\BL \to \OS{\ell}{\ell^\prime}$ is a pencil ball map with
     $\|f(0)\|<1$,
     then   there exists a contraction-valued analytic
     $\tilde{f}:\BL \to \OS{m}{m^\prime}$ such that
\beq\label{eq:bigDeal0fla}
f(x)= \cF_{f(0)}\big(\varphi(x) \big),
\eeq
where
\beq\label{eq:bigDeal1fla}
\varphi(x)=\cF_{f(0)}\big(f(x)\big)=
U
    \begin{pmatrix} \tL(x) & 0\\ 0&\tilde{f}(x)  \end{pmatrix} V^*
 \eeq
 for some $m,m'\in\NN_0$ and
 unitaries $U\in \CC^{\ell'\times \ell'}$ and $V\in\CC^{\ell\times \ell}$.
\end{corollary}

\subsection{Pencil ball to pencil ball maps}
 \label{sec:classifyPencilBall}
 Suppose $L,L'$ are homogeneous linear pencils. An analytic map
 $f:\cB_{L}\to\cB_{L'}$ is a {\bf pencil ball to pencil ball map} if
 $f(\partial\cB_{L})\subseteq \partial\cB_{L'}$.

\begin{corollary}\label{cor:pencil1}
    Suppose $L$ is a \nonsingular homogeneous linear pencil,
     $\tL$ is a minimal dimensional defining pencil for $\cB_L$,
    and let
    $L'$ be an arbitrary homogeneous linear pencil with $\CC^{\ell'\times \ell}$
    coefficients.
     If $f:\BL \to \cB_{L'}$ is a pencil ball to pencil ball map with
     $f(0)=0$, there exists
          a contraction-valued analytic
     $\tilde{f}:\BL \to \OS{m}{m^\prime}$ such that
 \begin{equation*}
    (L'\circ f)(x) =U \begin{pmatrix} \tL(x) & 0\\ 0&\tilde{f}(x)  \end{pmatrix} V^*
 \end{equation*}
 for some $m,m'\in\NN_0$ and unitaries $U\in \CC^{\ell'\times \ell'}$ and $V\in\CC^{\ell\times \ell}$.
\end{corollary}

\begin{corollary}\label{cor:pencil2}
    Suppose $L$, $L'$ are \nonsingular homogeneous linear pencils.
     If $f:\BL \to \cB_{L'}$ is a pencil ball to pencil ball map,
     then
\beq\label{eq:bigDeal3fla}
(L'\circ f)(x) = \cF_{L'\circ f(0)} \big(\varphi(x)\big),
\eeq
where
\beq\label{eq:bigDeal4fla}
\varphi(x)=\cF_{L'\circ f(0)}\big(L'\circ f(x)\big)
\eeq
is a pencil ball map $\cB_L\to\OS {\ell}{\ell'}$ mapping $0$ to $0$ and is therefore completely
described by Theorem {\rm\ref{thm:main}}.
\end{corollary}

It is clear that converses of Corollaries \ref{cor:pencil1} and \ref{cor:pencil2}
hold as well.

As a last result we show that origin-preserving
\emph{scalar}
analytic self-maps of $\cB_L$ are trivial.

\begin{corollary}\label{cor:autoPencil}
 Suppose $L$ is a \nonsingular homogeneous linear pencil.
     If $f:\BL \to \cB_{L}$ is a pencil ball to pencil ball map
with scalar coefficients and $f(0)=0$,
     then $f$ is linear.
\end{corollary}

\begin{proof}
We may assume without loss of generality $L$ is minimal.
Then by Corollary \ref{cor:pencil1}, 
$$
(L\circ f)(x)= U \begin{pmatrix} L(x) & 0\\ 0&\tilde{f}(x)  \end{pmatrix} V^*
$$
 for some unitaries $U, V$, and contraction-valued analytic
$\tilde f:\cB_L\to\OS {m}{m'}$. Comparing dimensions we see
$m'=m=0$, i.e., there is no $\tilde f$. Hence
$$
(L\circ f)(x)= U L(x) V^*.
$$
Since $L$ is \nonsingular this implies $f$ is linear.
\end{proof}

\section*{Acknowledgments}
The authors thank an anonymous referee for his or her comments.

\end{document}